\documentclass[11pt,a4paper]{amsart}
\usepackage[utf8]{inputenc}
\usepackage{amsmath}
\usepackage{amsfonts}
\usepackage{amssymb}
\usepackage{amsthm} 
\newtheorem{definition}{Definition}[section]
\newtheorem{theorem}[definition]{Theorem}
\newtheorem{proposition}[definition]{Proposition}
\newtheorem{corollary}[definition]{Corollary}
\newtheorem{lemma}[definition]{Lemma}
\newtheorem{remark}[definition]{Remark} 
\newcommand\cInd{\operatorname{c-Ind}} 
\def\tr{{\rm tr}\,} 
\def\GL{{\rm GL}} 
\def\M{{\rm M}}

\usepackage{xcolor}
\usepackage{hyperref} 
\begin{document} \title{Orbits of cuspidal types on $\GL_{p}(\mathcal{O}_{F})$} 
\author[A.Szumowicz]{Anna Szumowicz} 
\address{Caltech, The Division of Physics, Mathematics and Astronomy, 1200 E California Blvd, Pasadena CA 91125}  
\email{anna.szumowicz@caltech.edu}
\date{\today} 
\begin{abstract}Let $F$ be a non-Archimedean local field and let $\mathcal{O}_{F}$ be its ring of integers. The orbit of an irreducible representation $\rho$ of $\mathrm{GL}_n(\mathcal{O}_F)$ is a conjugacy class in $\mathfrak{gl}_n(\mathcal{O}_F)$ attached to $\rho$ by means of Clifford's theory. We give a description of orbits of cuspidal types on $\GL_{p}( \mathcal{O}_{F})$, with $p$ prime. We determine which of them are regular and we provide an example which shows that the orbit of a representation does not always determine whether it is a cuspidal type or not.\end{abstract}  
\maketitle

\begin{section}{Introduction}
 Let $G$ be a profinite group. A classical theorem of Clifford allows us to study the irreducible smooth representations of $G$ through their restrictions to open normal subgroups. When applied to $\mathrm{GL}_n(\mathcal{O}_F)$, where $F$ is a local non-Archimedean field and $\mathcal{O}_F$ is the ring of integers, this leads to the theory of orbits as introduced by Shintani \cite{Sh} and studied in \cite{H}, \cite{KOS}, \cite{SS}. 
 On the other hand we can look at the representations of $\GL_{n}(\mathcal{O}_{F})$ through the lens of the smooth representation theory of $\GL_{n}(F)$. The building blocks of the latter are given by the cuspidal representations. The cuspidal representations of $\GL_n(F)$ were parametrized by Bushnell and Kutzko \cite{BK} using cuspidal types, which form a special class of irreducible representations of open compact subgroups of $\GL_n(F)$. The main goal of this paper is to clarify the connection between these two points of view and describe the cuspidal types in terms of orbits. All representations we consider are smooth and over $\mathbb{C}$. 
\subsection{Orbits of irreducible representations of $\GL_{n}(\mathcal{O}_{F})$.}
The exposition below follows \cite{S}. Any irreducible smooth representation $\rho $ of $\GL_{n}(\mathcal{O}_{F})$ factors through a finite group $\GL_{n}(\mathcal{O}_{F}/\mathfrak{p}_{F}^{r})$ where $r\geqslant 1$ and $\mathfrak{p}_{F}$ is the maximal ideal of $\mathcal{O}_{F}$. The minimal natural number $r$ with this property is called the conductor of the representation $\rho $. Let $\rho $ be an irreducible smooth representation of $\GL_{n}(\mathcal{O}_{F})$ with conductor $r>1$. Sometimes it will be convenient to view $\rho$ as a representation of $\GL_{n}(\mathcal{O}_{F}/\mathfrak{p}_{F}^{r})$. In this case we will denote it by $\bar{\rho}$. Let $l=\lfloor \frac{r+1}{2}\rfloor$ and let $K^{l}$ be the kernel of the reduction map $\GL_{n}(\mathcal{O}_{F}/\mathfrak{p}_{F}^{r})\to \GL_{n}(\mathcal{O}_{F}/\mathfrak{p}_{F}^{l})$. Note that $K^{l}$ is an abelian group. We fix once and for all an additive character  $\psi :F\rightarrow \mathbb{C}^{\times }$ with conductor $\mathfrak{p}_{F}$, i.e. $\mathfrak{p}_{F}$ is the biggest fractional ideal of $F$ on which $\psi $ is trivial. Write $\M_{n}(R)$ for the set of $n\times n$ matrices with entries in a ring $R$. By Clifford's theorem (see \cite[6.2]{I}), \begin{equation}\label{Clifforddecom}
\bar{\rho}\mid _{K^{l}}=m\bigoplus _{\bar{\alpha } \sim \bar{\alpha _{1}}}\bar{\varphi}_{\bar{\alpha }},
\end{equation} where $\bar{\alpha}_{1}\in \M _{n}(\mathcal{O}_{F}/\mathfrak{p}_{F}^{r-l})$, $\bar{\alpha }$ runs over the conjugacy class of $\bar{\alpha _{1}}$ under $\GL_{n}(\mathcal{O}_{F}/\mathfrak{p}^{r-l})$, $m\in \mathbb{N}$ and the characters  $\bar{\varphi}_{\bar{\alpha}}:K^{l}\rightarrow \mathbb{C}^{\times }$ are defined as $$\bar{\varphi }_{\bar{\alpha}}(1+x)=\psi(\varpi_F^{-r+1}\tr (\widehat{\alpha}\widehat{x})),$$ for some lifts $\widehat{x}, \widehat{ \alpha }$ of $x,\bar{\alpha }$ to elements in $\M _{n}(\mathcal{O}_{F})$. The definition of $\bar{\varphi }_{\bar{\alpha }}$ does not depend on the choice of the lifts. If a matrix $\alpha\in \M _{n}(\mathcal{O}_{F})$ is such that its image in $\M _{n}(\mathcal{O}_{F}/\mathfrak{p}_{F}^{r-l})$ appears in the decomposition (\ref{Clifforddecom}) we say that $\alpha$ \textbf{is in the orbit of} $\rho$. We say that a representation is \textbf{regular} if its orbit contains a matrix whose image in $\M _{n}(\mathcal{O}_{F}/\mathfrak{p}_{F})$ has abelian centralizer in $\GL_{n}(\mathcal{O}_{F}/\mathfrak{p}_{F})$. The regular representations of $\GL_{n}(\mathcal{O} _{F})$ were introduced by Shintani \cite{Sh} and were rediscovered by Hill \cite{H}. Those are in certain sense the best behaved representations of $\GL_{n}(\mathcal{O} _{F})$. Krakovski, Onn and Singla \cite{KOS} constructed all such representations under the condition that the characteristic of the residue field of $F$ is odd. Stasinski and Stevens in \cite{SS} constructed all regular representations of $\GL_{n}(\mathcal{O}_{F})$.

\begin{subsection}{Cuspidal types}
Let us recall the definition of a cuspidal type. Let $n \in \mathbb{N}$, $n \geqslant 1$ and let $\pi$ be an irreducible cuspidal representation of $\GL_{n}(F)$. Let
\begin{equation} 
\label{definitioninertialsupporti} \mathfrak{I}(\pi)=\lbrace\pi' \mid\: \pi' \cong \pi \otimes\chi\circ {\rm det}\quad \text{for some unramified character}\: \chi\: \text{of}\: F^{\times }\rbrace 
\end{equation} be the inertial support of $\pi $.   
\begin{definition}Let $H$ be a compact open subgroup of $\GL_{n}(F)$ and let $\pi $ be an irreducible cuspidal representation of $\GL _{n}(F)$. We say that an irreducible smooth representation $\lambda $ of $H$ is a \textbf{type} on $H$ for $\mathfrak{I} (\pi )$ if the following condition is satisfied: for any irreducible smooth representation $\pi _{1}$ of $\GL_{n}(F)$ 
\begin{equation*} \pi _{1} \mid _{H} \quad \text{contains}\quad \lambda \quad  \text{if and only if}\quad \mathfrak{I}(\pi _{1})= \mathfrak{I}(\pi ). \end{equation*} \end{definition} 
In this paper we mostly consider types on $\GL_{n}(\mathcal{O}_{F})$ so we will supress $\GL_{n}( \mathcal{O}_{F})$ from the notation. We say that a representation is a cuspidal type when it is a type on $\GL _{n}( \mathcal{O}_{F})$ for $\mathfrak{I} ( \pi )$ for some irreducible cuspidal representation $\pi $ of $\GL _{n}(F)$. Henniart gave an explicit description of cuspidal types on $\GL_{2}(\mathcal{O}_{F})$ in \cite{Hen}. Bushnell--Kutzko's construction of irreducible cuspidal representations of $\GL_{n}(F)$ implies the existence of cuspidal types on $\GL _{n}(\mathcal{O} _{F})$ (see \cite{P}) for every $n\in\mathbb{N}$, $n\geqslant 2$. Finally, Paskunas \cite{P} proved that for any irreducible cuspidal representation $\pi $ of $\GL_{n}(F)$ there exists a unique up to isomorphism irreducible smooth representation $\lambda$ of $\GL _{n} ( \mathcal{O} _{F})$ depending only on $\mathfrak{I}(\pi )$ which is a cuspidal type on $\GL _{n}( \mathcal{O} _{F})$ for $\mathfrak{I}(\pi )$. Using that and the local Langlands correspondence he deduced the inertial Langlands correspondence between cuspidal types on $\GL_n(\mathcal O_F)$ and certain irreducible representations of the inertia group of $F$ (see \cite[Corollary 1.4]{P}).


\end{subsection}

\begin{subsection}{Cuspidal types in terms of orbits} \label{introductionorbits}  
In \cite{S} Stasinski asked which cuspidal types are regular. Our main results answer this question when $n$ is a prime number.
We classify all orbits which can give cuspidal types on $\GL_{p}(\mathcal{O}_{F})$ with conductor at least $4$ and we precisely determine orbits of cuspidal types in small conductor case for $p=2$. The proof relies on Bushnell-Kutkzo's theory, which parametrizes the irreducible cuspidal representations of $\GL_{n}(F)$ in terms of strata (see \cite{BK}). We translate this description into the (arguably simpler) language of orbits. Some information is lost, but for high conductors we give a full characterization of cuspidal types in terms of orbits when $n$ is prime.   

Recall that a polynomial $x^{n}+a_{n-1} x^{n-1}+ \ldots +a_{0}\in\mathcal{O}_{F}[x]$ is called Eisenstein if $a_{1}, \ldots , a_{n-1} \in \mathfrak{p}_{F}$ and $a_{0} \in \mathfrak{p}_{F}\setminus \mathfrak{p}_{F}^{2}$.
For $i\in\mathbb{N}$, $i\geq 1$ let $$U_{\mathfrak{M}}^{i}:=
\begin{pmatrix}1+\mathfrak{p}_{F}^{i}& \mathfrak{p}_{F}^{i}\\
\mathfrak{p}_{F}^{i}&1+\mathfrak{p}_{F}^{i}
\end{pmatrix}.$$ For a matrix $\alpha \in\varpi_{F}^{-r+1}\M_{2}(\mathcal{O}_{F})$, $r\in\mathbb N$, $r\geq 1$, we define the character $\psi_{\alpha}\colon U_{\mathfrak{M}}^{\lfloor \frac{r+1}{2}\rfloor}\to\mathbb{C}^{\times}$: $$\psi_{\alpha}(1+x):=\psi(\tr{\alpha x}).$$
Finally, let $k_F:=\mathcal O/\mathfrak{p}_F$.
\begin{theorem}
\label{TheoremGL2i}
A cuspidal type on $K_{2}:=\GL_{2}( \mathcal{O}_{F})$ is precisely a one-dimensional twist of one of the following:
\begin{enumerate}
\item a representation inflated from some irreducible cuspidal representation of $\GL_{2}(k_{F})$; 
\item a representation whose orbit contains a matrix whose characteristic polynomial is irreducible mod $\mathfrak{p}_{F}$;
\item a representation whose orbit contains a matrix $\beta _{0} $ whose characteristic polynomial is Eisenstein and which satisfies one of the following: \begin{enumerate}
\item it has conductor at least $4$;
\item it has conductor $r=2$ or $3$ and is isomorphic to $\mathrm{Ind}_{S}^{K_{2}}\theta $ where $S=\bigcup_{a\in\mathcal{O}_{F}^{\times}}
\begin{pmatrix}
a+\mathfrak{p}_{F}& \mathcal{O}_{F}\\
\mathfrak{p}_{F}& a+\mathfrak{p}_{F}
\end{pmatrix}$, $\theta$ is an irreducible character of $S$ such that $\theta \mid _{ U _{ \mathfrak{M}}^{ \lfloor \frac{r+1}{2} \rfloor}}=m \psi _{\varpi_{F}^{-r+1}\beta_{0} }$ for certain $m \in \mathbb{Z}$ and $ \theta $ does not contain the trivial character of the group $\begin{pmatrix}
1&\mathfrak{p}_{F}^{r-2}\\
0&1
\end{pmatrix}$.   
\end{enumerate}  
\end{enumerate}
\end{theorem} 
To state the version for $\GL_p(\mathcal{O}_{F})$, $p\geq 2$ we need one more piece of notation. Let $\mathfrak{I}$ be the $\mathcal{O}_{F}$-order consisting of matrices that are upper triangular modulo $\mathfrak{p}_{F}$. Let $U_{\mathfrak{I}}$ be the group of  invertible elements of $\mathfrak{I}$ and let $\mathfrak{P}_{\mathfrak{I}}$ be the Jacobson radical in $\mathfrak{I}$. We choose $\Pi _{\mathfrak{I}}$ such that $\Pi _{\mathfrak{I} }\mathfrak{I}=\mathfrak{P}_{\mathfrak{I}}$ (see Lemma \ref{principal} and the following discussion). 
\begin{theorem} \label{TheoremGLpi}
If $\lambda $ is a cuspidal type on $K=\GL_{p}( \mathcal{O}_{F})$, then it is a one-dimensional twist of one of the following:
\begin{enumerate} \item a representation which is inflated from an irreducible cuspidal representation of $\GL_{p}(k_{F})$;
\item a representation whose orbit contains a matrix whose characteristic polynomial is irreducible modulo $ \mathfrak{p}_{F}$; 
\item a representation whose orbit contains a matrix of the form $\Pi _{ \mathfrak{I}}^{j} B$ where $0<j<p$ and $B \in U_{ \mathfrak{I}}$.\end{enumerate}
Moreover, if $\pi$ is an irreducible cuspidal representation of $G$ such that $l(\pi)\leqslant l(\chi\pi)$ for all one-dimensional characters $\chi:F^{\times }\to \mathbb{C}^{\times}$ \footnote{for the definition of the level $l(\pi)$ see subsection \ref{simplestrata}}, then the cuspidal type for $\mathfrak{I}(\pi)$ is of the form $(1)$, $(2)$ or $(3)$ (no twisting is needed).
Conversely, if a representation is a one-dimensional twist of a representation of the form $(3)$ and has conductor at least $4$, or is of the form $(1)$ or $(2)$, then it is a cuspidal type. \end{theorem} 
We give an example (see section \ref{sectionexample})which shows that for low conductor orbit does not determine whether a representation is a cuspidal type or not.
Representations of $\GL_p(\mathcal{O}_F)$ whose orbit contains a matrix whose characteristic polynomial is irreducible modulo $\mathfrak{p}_{F}$ are regular. In Subsection \ref{regularity} we prove that a matrix of the form $\Pi _{\mathfrak{I}}^{j} B$ with $0<j<p$ and $B\in U _{ \mathfrak{I}} $ is regular if and only if $j=1$. The characteristic polynomial of a matrix of the form $\Pi _{ \mathfrak{I}}B$ is Eisenstein but this is not the case for the matrices of the form $\Pi _{\mathfrak{I}}^{j}B$ with $1<j<p$. Therefore a cuspidal type on $\GL_{p}(\mathcal{O}_{F})$ of conductor $r\geq 4$ is regular if and only if its orbit contains a matrix whose characteristic polynomial is irreducible modulo $\mathfrak{p}_{F}$ or a matrix whose characteristic polynomial is Eisenstein. In particular, for $p>2$, there are cuspidal types which are not regular. Indeed, if a representation has conductor at least $4$ and is of the form $(3)$ from Theorem \ref{TheoremGLpi} with $j>1$ then it is a cuspidal type but it is not regular.   \end{subsection}

\begin{subsection}{Outline of the paper}

In \textbf{Section \ref{sectionstrata}} we recall the properties of hereditary orders and simple strata,  restricting our attention to the case of $\GL_{p}(F)$ with $p$ prime. In \ref{simplestrata} we give an explicit description of simple strata which is one of the crucial ingredients in the proof of Theorem \ref{TheoremGLpi}. In \ref{sectionclasscusp} we recall the classification of irreducible cuspidal representations of $\GL_{p}(F)$. In \ref{orbits} we recall the Clifford's theorem and its conseqeunces. 

In \textbf{Section \ref{sectionproof}} we prove Theorem \ref{TheoremGL2i} and Theorem \ref{TheoremGLpi}. Then we determine which of the cuspidal types on $\GL _{p} (\mathcal{O} _{F})$ are regular.  

In \textbf{Section \ref{sectionexample}} we give an example of two representations of $\GL_{2}( \mathcal{O}_{F})$ with share the same orbit but one is  a cuspidal type and the other is not.       
      \end{subsection}

\begin{subsection}{Notation}We will write $\lfloor a \rfloor $ for the biggest integer less  than or equal to $a$ and $\tr M$ for the trace of a matrix $M$. For any local non-Archimedean field $E$ we will denote by $\mathcal{O}_{E}$ its ring of integers, by $\mathfrak{p}_{E}$ the maximal ideal in $\mathcal{O}_{E}$, by $\varpi_{E}$ a prime element in $E$, by $\mathcal{O}_{E}^{\times }$ the group of invertible elements of $\mathcal{O}_{E}$ and by $k_{E}$ the residue field of $E$. We fix a non-Archimedean local field $F$ and a prime number $p$ (which is not necessarily the characteristic of the residue field). Let $E/F$ be a finite field extension. Write $e(E/F)$ for the ramification index and $f(E/F)$ for the residue class degree. Let $G:= \GL _{p}(F)$ and $K:= \GL _{p}( \mathcal{O} _{F})$. We write $Z$ for the center of $G$. Let $V$ be a vector space over $F$ of dimension $p$ and $A:={\rm End}_{F}(V)$. For a local field $E$ we denote by $\nu _{E}$ the additive valuation on $E$ which takes value $1$ on a uniformizer. Write $\pi $ for a representation of $\GL_{n}(F)$ and let $\chi $ be a character of $F^{\times}$. Set $\chi \pi:=( \chi \circ {\rm det})\otimes \pi$. If $H$ is a subgroup of $G$, we denote by $N_{G}(H)$ the normalizer of $H$ in $G$. Write $\beta^{g}=g^{-1}\beta g$ for any matrix $\beta\in \textrm{M}_{p}(F)$ and $g\in G$. For a character $\bar{\varphi}_{\bar{\alpha}}$ on $K^{l}$ let ${\rm Stab}_{\GL _{2}(\mathcal{O} _{F})}\bar{\varphi}_{\bar{\alpha}}$ be the preimage of ${\rm Stab}_{\GL_{2}(\mathcal{O}_{F}/\mathfrak{p}^{r})}\bar{\varphi}_{\bar{\alpha}}$ through the canonical projection $\GL _{2}( \mathcal{O} _{F})\twoheadrightarrow \GL_{2}(\mathcal{O}_{F}/\mathfrak{p}^{r})$.    \end{subsection} \end{section} 

\section*{Acknowledgements}I am grateful to Alexander Stasinski, Anne- Marie Aubert, Uri Onn, Shaun Stevens, Peter Latham, Arnaud Mayeux for helpful discussions. The author was supported by UK EPSRC Doctoral Studentship. The work was done partially while the author was visiting the Institute for Mathematical Sciences, National University of Singapore in 2018-2019.            

\begin{section}{Simple strata and cuspidal representations} 
\label{sectionstrata}  

\begin{subsection}{Cuspidal types on $K$}
Paskunas in \cite{P} has proven the uniqueness of cuspidal types:

\begin{theorem}[cf \cite{P}, Theorem 1.3 ]\label{Paskunas}Let $\pi $ be an irreducible cuspidal representation of $G$. Then there exists a smooth irreducible representation $\rho$ of $K$ depending on $\mathfrak{I}(\pi )$, such that $\rho $ is a cuspidal type on $K$ for $\mathfrak{I}(\pi )$. Moreover, $\rho $ is unique (up to isomorphism) and it occurs in $\pi\mid _{K}$ with multiplicity 1.   
\end{theorem}
 
 Denote by $X_{F}(G)$ the group of $F$-rational characters of $G$. Denote by $\parallel \cdot \parallel_{F}$ normalized absolute value \footnote{If $\varpi_{F}$ is a uniformizer in $F$ then we want $\|\varpi_{F}\|_F=|k_F|^{-1}$.} on $F$. Define \begin{center}
$^{\circ }G=\bigcap _{\phi\in X_{F}(G)} {\rm Ker}(\parallel \phi \parallel _{F})$.
\end{center} The following proposition will be a useful tool while describing cuspidal types in terms of orbits. 
\begin{proposition}[\cite{BK}, 5.4 Proposition]
\label{BKprop}
Let $J$ be an open compact mod $Z$ subgroup of $G$. Let $\pi$ be an irreducible cuspidal representation of $G$ of the form $\pi \cong \cInd _{J}^{G}\tau _{1}$ for some representation $\tau _{1}$ of $J$. Let $J^{\circ}=J\cap {}^{\circ}G$ and let $\tau $ be an irreducible component of $\tau _{1}\mid_{J^{\circ}}$. Then $J^{\circ}$ is the unique maximal compact subgroup of $J$ and $\tau $ is a cuspidal type on $J^{ \circ }$ for $\mathfrak{I}(\pi )$.
\end{proposition}

\begin{remark}[\cite{BK}]\label{remarkBK}
Every irreducible cuspidal representation of $\GL_{p}(F)$ is isomorphic to one of the form as in Proposition \ref{BKprop}.
\end{remark} 
\begin{remark} 
Theorem \ref{Paskunas}, Proposition \ref{BKprop} and Remark \ref{remarkBK} do not use the assumption that $p$ is prime. 
\end{remark}  
\end{subsection}
\begin{subsection}{Hereditary orders}\label{chain1}
In this section we recall basic notions associated to hereditary orders in $A$. The given description of principal orders relies on the fact that $V$ is of a prime dimension. In the general case things are more complicated. For more detailed discussion on hereditary orders we refer to (\cite{BK}, 1.1). We call a finitely generated $\mathcal{O}_{F}$-submodule of $V$ containing an $F$-basis of $V$ an \textit{$\mathcal{O}_{F}$-lattice} in $A$. An \textit{$\mathcal{O}_{F}$-order} in $A$ is an $\mathcal{O}_{F}$-lattice in $A$ which is also a subring of $A$ (with the same identity element). A sequence $\mathcal{L}=\lbrace L_{i}:\small i\in\mathbb{Z}\rbrace$ of $\mathcal{O}_{F}$-lattices satisfying the following conditions:\begin{enumerate}\item $L_{i+1}\varsubsetneq L_{i}$, $i \in \mathbb{Z}$
\item there exists $e \in \mathbb{Z}$ such that $\mathfrak{p}_{F}L_{i}=L_{i+e}$ for every $i \in \mathbb{Z}$
\end{enumerate} is called an \textit{$\mathcal{O}_{F}$-lattice chain} in $V$. The number $e=e(\mathcal{L})=:e(\mathfrak{A}(\mathcal{L}))$ is uniquely determined for a given $\mathcal{L}$ and is called the \textit{ $\mathcal{O}_{F}$-period} of $\mathcal{L}$. For $n \in \mathbb{Z}$ and an $\mathcal{O}_{F}$-lattice chain $\mathcal{L}$ define \begin{equation*}
{\rm End}_{\mathcal{O}_{F}}^{n}(\mathcal{L})=\lbrace g \in A: \small gL_{i}\subseteq L_{i+n}, \small i\in \mathbb{Z} \rbrace .
\end{equation*}
Taking $n=0$ we get ${\rm End}_{\mathcal{O}_{F}}^{0}(\mathcal{L}) =:\mathfrak{A} (\mathcal{L})=\mathfrak{A}$, which is an $\mathcal{O}_{F}$-order in $A$. This is a hereditary order \cite[1.1.]{BK}. A hereditary order $\mathfrak{A}(\mathcal{L})$ is called principal if ${\rm dim}_{k_{F}}(L_{i}/L_{i+1})={\rm dim}_{k_{F}}(L_{j}/L_{j+1})$ for every $i,j \in \mathbb{Z}$. Let $\mathcal{L}$ be an $\mathcal{O}_{F}$-lattice chain. Then ${\rm End}_{\mathcal{O}_{F}}^{1}$ is the Jacobson radical of $\mathfrak{A}(\mathcal{L})$. We denote it by $\mathfrak{P}_{\mathfrak{A}}$ or by $\mathfrak{P}$ if the order is clear from the context. It is an invertible fractional ideal and we have \begin{equation*}
\mathfrak{P}^{n}_{\mathfrak{A}}= \mathfrak{P}^{n}={\rm End}_{\mathcal{O}_{F}}^{n}(\mathcal{L}) \quad \text{for any } n \in \mathbb{Z}. 
\end{equation*}We also have $\mathfrak{p}_{F}\mathfrak{A}=\mathfrak{P} _{\mathfrak{A}} ^{e(\mathfrak{A})}$. We denote $U(\mathfrak{A})=U^{0}_{\mathfrak{A}}$ the group of invertible elements in $\mathfrak{A}$ and we define the subgroups \begin{equation*}
U^{n}_{\mathfrak{A}}=1+\mathfrak{P}_{\mathfrak{A}}^{n} \quad \text{for any } n\in \mathbb{N}, \small n\geqslant 1.
\end{equation*} We define the \textit{normalizer} of $\mathfrak{A}$ as \begin{equation*}
\mathfrak{K}(\mathfrak{A})=\lbrace g \in G: \small g \mathfrak{A} g^{-1}=\mathfrak{A} \rbrace  
\end{equation*} or equivalently if $ \mathfrak{A}= \mathfrak{A}(\mathcal{L} )$ for some $\mathcal{O}_{F}$-lattice chain $\mathcal{L}$ as \begin{equation*} \mathfrak{K}(\mathfrak{A})= \lbrace g \in G : \small gL \in \mathcal{L} \quad \text{for any} \quad L \in \mathcal{L} \rbrace . \end{equation*}  We now restrict our attention to principal orders. 
\begin{lemma}\label{principal} 
Any principal order in $A$ is $\GL_{p}(F)$-conjugate to $\mathfrak{M}=\M _{p}(\mathcal{O}_{F})$ or to the order $\mathfrak{I}$ which consists of matrices with coefficients in $\mathcal{O}_{F}$ which are upper triangular modulo $\mathfrak{p}_{F}$:\begin{equation*}\mathfrak{M}=\begin{pmatrix}\mathcal{O}_{F}& \cdots & \cdots  & \mathcal{O}_{F}\\
\vdots & \cdots & \cdots  & \vdots \\
\vdots & \cdots & \cdots & \vdots \\
\mathcal{O}_{F} & \cdots & \cdots & \mathcal{O}_{F}
\end{pmatrix} 
\quad \text{and} \quad 
\mathfrak{I}=\begin{pmatrix}
\mathcal{O}_{F} & \cdots & \cdots & \mathcal{O}_{F}\\
\mathfrak{p}_{F} & \ddots & \ddots & \vdots \\
\vdots &\ddots &\ddots & \vdots \\
\mathfrak{p}_{F} & \cdots &\mathfrak{p}_{F} & \mathcal{O}_{F}
\end{pmatrix}. \end{equation*}\end{lemma}\begin{proof} The proof is based on the notion of an $\mathcal{O}_{F}$-basis of an $\mathcal{O}_{F}$-lattice chain. Let $\mathcal{L}=\lbrace L_{i}: \small i \in \mathbb{Z} \rbrace$ be an $\mathcal{O}_{F}$-lattice chain in $V$. An \textit{ $\mathcal{O}_{F}$-basis} of $\mathcal{L}$ is an $F$-basis $\lbrace v_{1}, \ldots , v_{p} \rbrace$ of $V$ such that it is an $\mathcal{O}_{F}$-basis of some $L_{j} \in \mathcal{L}$ and for every $i\in \mathbb Z$ there are (uniquely determined) integers $f(i,1)\leqslant f(i,2) \leqslant \ldots \leqslant  f(i,p)$ such that \[ L_{i}=\sum _{l=1}^{p} \mathfrak{p}_{F}^{f(i,l)}v_{l}.\]
Any $\mathcal{O}_{F}$-lattice chain has an $\mathcal{O}_{F}$-basis (\cite{BK}, 1.1). 

Let $\mathfrak{A}=\mathfrak{A}(\mathcal{L})$ be a principal order for an $\mathcal{O}_{F}$-lattice chain $\mathcal{L}= \lbrace L_{i}: \small i \in \mathbb{Z} \rbrace $. We want to show that $\mathfrak{A}$ is $\GL_{p}(F)$-conjugate to $\mathfrak{M}$ or $\mathfrak{I}$.  Let $\lbrace v_{1}, \ldots , v_{p} \rbrace$ be an $\mathcal{O}_{F}$-basis of $\mathcal{L}$. We use this basis to identify $A$ with $\M _{p}(F)$. Let $\mathcal{L}_{\mathrm{max}}$ be the $\mathcal{O}_{F}$-lattice chain formed by $\mathcal{O}_{F}$-lattices of the form \begin{equation*} \mathfrak{p}_{F}^{j}( \mathcal{O}_{F}v_{1}+ \ldots + \mathcal{O}_{F}v_{l} + \mathfrak{p}_{F}v_{l+1}+ \ldots +\mathfrak{p}_{F}v_{p}) \end{equation*} where $1 \leqslant l \leqslant p  $ and $j \in \mathbb{Z}$. The $\mathcal{O}_{F}$-lattice chain $\mathcal{L}$ is contained in $\mathcal{L}_{\mathrm{max}}$ (see \cite[1.1]{BK}). Since $\mathcal{L}$ is principal, ${\rm dim}_{k_{F}}(L_{i}/L_{i+1})= {\rm dim} _{k_{F}}(L_{l} /L_{l+1})$ for any $i,l \in \mathbb{Z}$. We have \[e(\mathcal L){\rm dim} _{k_{F}}(L_{1} /L_{2})=\sum _{i=1}^{e(\mathcal{L})} {\rm dim}_{k_{F}}(L_{i}/L_{i+1})={\rm dim}_{k_F}(L_1/\frak p_F L_1)=p.\] Therefore $e(\mathcal{L})=1 $ or $p$. 
If $e(\mathcal{L})=1$, then $\mathcal{L}$ consists of $\mathcal{O}_{F}$-lattices of the form $\mathfrak{p}_{F}^{j}(\mathcal{O}_{F} v_{1}+ \ldots + \mathcal{O}_{F}v_{p})$ for $j \in \mathbb{Z}$ and $\mathfrak{A}=\mathfrak{M}$. If $e(\mathcal{L})=p$, then $\mathcal{L}=\mathcal{L}_{\mathrm{max}}$ and $\mathfrak{A}(\mathcal{L})=\mathfrak{I}$.                                 
\end{proof}

Denote \begin{equation*}\Pi_{\mathfrak{M}}=\varpi _{F} {\rm Id} _{p \times p} \quad \text{and} \quad \Pi_{\mathfrak{I}}=\begin{pmatrix}0&1&0& \cdots &0\\
\vdots & \ddots & \ddots & \ddots & \vdots \\
\vdots & \ddots & \ddots & \ddots & \vdots \\
0 & \ddots & \ddots  & \ddots & 1 \\
\varpi _{F}& 0& \cdots & \cdots &0
\end{pmatrix} \end{equation*} where ${\rm Id} _{p \times p}$ denotes the identity matrix of size $p \times p$. One can verify that 
$\mathfrak{P}_{\mathfrak{M}}=\Pi _{\mathfrak{M}} \mathfrak{M}= \mathfrak{M} \Pi _{ \mathfrak{M}} $ and $\mathfrak{P}_{\mathfrak{I}}= \Pi _{\mathfrak{I}} \mathfrak{I}= \mathfrak{I} \Pi _{\mathfrak{I}}$. \begin{corollary}\label{primeelement}  
For a principal order $\mathfrak{A}$ there exists an element $\Pi _{\mathfrak{A}}$ such that $\mathfrak{P}_{\mathfrak{A}}=\Pi _{\mathfrak{A}} \mathfrak{A} =  \mathfrak{A} \Pi _{\mathfrak{A}}$.
\end{corollary}\begin{proof} By Lemma \ref{principal} it is enough to check the statement for $\mathfrak{M}$ and $\mathfrak{I}$. 
\end{proof}We call an element $\Pi _{\mathfrak{A}}$ from Corollary \ref{primeelement} a prime element in $\mathfrak{A} $. 
The normalizer $\mathfrak{K}( \mathfrak{A})$ is an open compact modulo center subgroup of $G$ (see \cite{BK}, section 1.1). Define \begin{equation*} \nu _{\mathfrak{A}}(a):= {\rm max} \lbrace n \in \mathbb{Z}: \small a\in \mathfrak{P}_{\mathfrak{A}}^{n} \rbrace ,\quad\quad a\in A \end{equation*} and let $\nu _{\mathfrak{A}}(0)=\infty$.  
 
 \begin{lemma}\label{normalizer}
 Let $\mathfrak{A}$ be a principal order. Then, $\mathfrak{K}(\mathfrak{A})=U_{\mathfrak{A}}\rtimes\langle \Pi _{\mathfrak{A}} \rangle$.
 \end{lemma}  
 \begin{proof}
Without loss of generality we can assume that $\mathfrak A=\mathfrak M$ or $\mathfrak I$. The subgroup $U_\mathfrak A$ is always normal in $\mathfrak K(\mathfrak A)$. The quotient $\mathfrak K(\mathfrak A)/U_\mathfrak A$ is identified with $\mathbb Z$ via the map $\nu_\mathfrak A$, under which $\langle \Pi_\mathfrak A\rangle$ is mapped bijectively to $\mathbb Z$.
 \end{proof} 

\end{subsection}\begin{subsection}{Simple strata}\label{simplestrata} 

A simple stratum is a notion used in the classification of irreducible cuspidal representations of $\GL_{p}(F)$ in \cite{BK}. We focus on simple strata which come from principal orders as these are the ones relevant in the classification of irreducible cuspidal representations of $\GL_{p}(F)$ for $p$ prime. Again, the given properties rely on the fact that the dimension of $V$ is prime. We use a definition (Definition \ref{definitionsimplestratum}) of a simple stratum which is not a standard one (comes from \cite{BKnotes}) but we prove that in the relevant cases it is equivalent with the one used in \cite{BK}. The goal of this subsection is to prove Proposition \ref{Propositionsimplestratum}. 
\begin{definition}A $4$-tuple $[\mathfrak{A},n,r,\beta ]$ is called a stratum in $A$ if $\mathfrak{A}$ is a hereditary $\mathcal{O}_{F}$-order in $A$, $n, r$ are integers such that $n>r$ and $\beta \in A$ is such that $\nu_{\mathfrak{A}}(\beta )\geqslant -n$. 
\end{definition} 
We say that two strata $[ \mathfrak{A}_{1}, n_{1}, r_{1}, \beta _{1} ]$ and $[ \mathfrak{A}_{2}, n_{2}, r_{2}, \beta _{2} ]$ are equivalent if \begin{equation*} \beta _{1}+ \mathfrak{P}_{1}^{-r_{1}}= \beta _{2}+ \mathfrak{P}_{2}^{-r_{2}} \end{equation*} where $\mathfrak{P}_{1}$ (resp. $\mathfrak{P}_{2}$) is the Jacobson radical of $\mathfrak{A}_{1}$ (resp. $\mathfrak{A}_{2}$). We denote it by $[\mathfrak{A}_1, n_{1},r_{1}, \beta_1]\sim [\mathfrak{A}_2,n_2, r_2, \beta _2]$. If $n>r \geqslant \lfloor \frac{n}{2} \rfloor \geqslant 0$, then we can associate with a stratum $[ \mathfrak{A} , n,r, \beta ]$ a character $\psi _{\beta }: U_{\mathfrak{A}}^{r+1} \rightarrow \mathbb{C}^{ \times } $ which is trivial on $U_{\mathfrak{A}}^{n+1}$ and defined by $\psi _{\beta  } (x)= \psi (\tr ( \beta  (x-1)))$. We say that an irreducible smooth representation $ \pi $ of $\GL_{p}(F)$ contains a stratum $ [\mathfrak{A} , n, r, \alpha ]$ if $ \pi $ contains the character $ \psi _{ \alpha }$ of $U_{\mathfrak{A}}^{r+1}$. We define the normalized level of an irreducible smooth representation $\pi $ as \begin{align*} l( \pi ) := {\rm min} \bigg\lbrace \frac{n}{e( \mathfrak{A})}: \mathfrak{A} \textrm{ is a hereditary order, } n \in \mathbb{N}, \large n \geqslant 0 \\ \textrm{and }  \pi \textrm{ contains a trivial character of }  U_{\mathfrak{A}}^{n+1} \bigg\rbrace . \end{align*} 
We say that a stratum $[\mathfrak{A},n,n-1, \beta ]$ is fundamental if $\beta + \mathfrak{P}_{ \mathfrak{A}}^{1-n}$ does not contain a nilpotent element of $A$. We say that two strata $[ \mathfrak{A} _{1}, n_{1}, r_{1}, \beta _{1} ]$ and $[ \mathfrak{A}_{2}, n_{2}, r_{2}, \beta _{2} ]$ intertwine in $G$ if there exists $x \in G$ such that $x (\beta _{2} + \mathfrak{P}_{ \mathfrak{A}_{2}}^{-r_{2}})x^{-1} \cap (\beta _{1}+ \mathfrak{P}_{ \mathfrak{A}_{1}}^{-r_{1}}) \neq \emptyset$. 

Let $H_{1}, H_{2}$ be two compact open subgroups of $G$ and let $\pi _{1}$ (resp. $\pi _{2}$) be an irreducible smooth representation of $H_{1}$ (rep. $H_{2}$). Take $g \in G$. Write $H_{1}^{g}:= g^{-1}H_{1}g$. Define $\pi _{1}^{g}$ to be a representation of $H_{1} ^{g}$ such that $\pi _{1}^{g}(h)=\pi _{1} (ghg^{-1})$ for any $h \in H_{1}^{g}$. We say that  $g \in G$ intertwines $\pi _{1}$ with $\pi _{2}$ if $ {\rm Hom}_{H_{1}^{g} \cap H_{2}}(\pi _{1}^{g}, \pi _{2}) \neq 0$.   

\begin{lemma}{(see \cite[11.1 Proposition 1]{BH} and \cite[Lemma 1.13.5 ]{Alm})  }\label{fundamentalintertwine} Let $\pi $ be an irreducible cuspidal representation of $G$. Let $\mathfrak{A}_{1}$, $\mathfrak{A}_{2}$ be principal orders and let $[ \mathfrak{A}_{1}, n_{1}, n_{1}-1, \beta _{1} ]$ and $[ \mathfrak{A}_{2}, n_{2}, n_{2}-1, \beta _{2} ]$ be two strata contained in $\pi $. Then they intertwine. \end{lemma} 

In order to introduce the notion of a simple stratum we first recall the definition of a minimal element over $F$. 
 \begin{definition}Let $E/F$ be a finite field extension with $E=F[\beta]$. We say that $\beta $ is minimal over $F$ if the following is satisfied:
\begin{itemize}
\item $gcd(\nu _{E}(\beta), e(E/F))=1$ and
\item $\varpi_{F}^{-\nu _{E}(\beta)}\beta^{e(E/F)}+\mathfrak{p}_{E}$ generates the extension of the residue fields $k_{E}/k_{F}$.
\end{itemize}
\end{definition}

\begin{definition}{(\cite[7.11]{BKnotes})} \label{definitionsimplestratum}  A stratum $[\mathfrak{A}, n, n-1, \beta ]$ is called simple if \begin{enumerate} \item  $E=F[ \beta ]$ is a field
\item $\beta \mathfrak{A} =\mathfrak{P}_{\mathfrak{A}}^{-n}$ 
\item $ \beta $ is minimal over $F$ \end{enumerate} \end{definition} 
Let $\pi $ be an irreducible cuspidal representation of $G$ which contains a simple stratum $[\mathfrak{A},n, n-1, \beta ]$. Since we consider $\GL_{p}(F)$ with $p$ prime there are only two possibilities for the degree $[F[\beta ]:F]$. Namely it is $1$ or $p$.
\begin{lemma}{(cf. \cite[1.5.6 Exercise]{BK})}\label{normalizercontainsE} Let $[ \mathfrak{A} , n, n-1, \beta ]$ be a simple stratum with $\mathfrak{A}=\mathfrak{M}$ or $\mathfrak{I}$. Denote $E=F [ \beta ]$. Then $E^{ \times } \subseteq \mathfrak{K}( \mathfrak{A}) $. \end{lemma} \begin{proof} First we prove that $\beta \in \mathfrak{K}( \mathfrak{A})$. Consider an $\mathcal{O}_{F}$-lattice chain $\mathcal{L}= \lbrace L_{i}: \small i \in \mathbb{Z} \rbrace $ such that $\mathfrak{A}=\mathfrak{A} (\mathcal{L})$. Take arbitrary $L_{i} \in \mathcal{L}$. We want to have $\beta L_{i} \in \mathcal{L}$. The fractional ideal $\mathfrak{P}_{\mathfrak{A}}^{n}$ is invertible and $\mathfrak{P}_{\mathfrak{A}}^{n}\mathfrak{P}_{\mathfrak{A}}^{-n}=\mathfrak{A}$. We have 
\begin{equation*} L_{i-n} = \mathfrak{A} L_{i-n}= \mathfrak{P} _{\mathfrak{A}}^{-n}\mathfrak{P}_{\mathfrak{A}}^{n} L_{i-n} \subseteq \mathfrak{P}_{\mathfrak{A}}^{-n} L_{i}= \beta \mathfrak{A} L_{i} = \beta  L_{i} \subseteq L_{i-n}. \end{equation*} 
Therefore $ \beta L_{i} = L_{i-n} \in \mathcal{L}$ for any $i \in \mathbb{Z}$ and $\beta \in \mathfrak{K}(\mathfrak{A})$. 

Since $\beta $ is minimal over $F$ the value $\nu _{E} (\beta )$ is coprime with $e(E/F)$. Therefore there exist $n_{1},n_{2} \in \mathbb{Z}$ such that $1=n_{1}\nu _{E} (\beta) +n_{2}e(E/F)= \nu _{E}( \beta ^{n_{1}} \varpi _{F}^{n_{2}})$. We can write any element from $E^{\times }$ as $u (\beta ^{n_{1}} \varpi _{F}^{n_{2}})^{m}$ for some $u \in \mathcal{O} _{E}^{ \times}$, $m \in \mathbb{Z}$. Since $\beta ^{n_{1}} \varpi _{F}^{n_{2}} \in \mathfrak{K}( \mathfrak{A})$, to finish the proof it is enough to show that $\mathcal{O}_{E}^{ \times } \subseteq \mathfrak{K}(\mathfrak{A})$. First we want to show that $\mathcal{O}_{E} \subseteq \mathfrak{A}$. 

By the definition of a minimal element, $\varpi _{F}^{-\nu _{E}( \beta )} \beta ^{e(E/F)}+ \mathfrak{p}_{E}$ generates $k_{E}/k_{F}$ so $\mathcal{O}_{E}=\mathcal{O}_{F}[\varpi _{F} ^{- \nu _{E} (\beta )} \beta ^{e(E/F)} ] +\varpi_{E}\mathcal{O}_{E}$. Iterating
 \begin{align*} \mathcal{O}_{E} =& \mathcal{O}_{F} [ \varpi _{F} ^{- \nu _{E} ( \beta ) } \beta ^{e(E/F)} ] + \varpi _{E} \mathcal{O}_{E} = \mathcal{O}_{F} [ \varpi _{F} ^{- \nu _{E} ( \beta )} \beta ^{e(E/F)}] +\\ &\varpi _{E} \mathcal{O}_{F} [ \varpi _{F} ^{- \nu _{E} ( \beta )} \beta ^{e(E/F)}] + \ldots + \varpi _{E} ^{e(E/F)-1} \mathcal{O}_{F}[ \varpi _{F} ^{- \nu _{E} ( \beta )} \beta ^{e(E/F)}] +\mathfrak{p}_{F} \mathcal{O}_{E}. \end{align*} By Nakayama's lemma, 
 
\begin{align*} \mathcal{O}_{E}=& \mathcal{O}_{F} [\varpi _{F}^{- \nu _{E} ( \beta )} \beta ^{e(E/F)} ] + \varpi _{E} \mathcal{O}_{F} [\varpi _{F} ^{- \nu _{E}(\beta )} \beta ^{e (E/F)}] + \ldots \\+& \varpi _{E} ^{e(E/F)-1} \mathcal{O}_{F}[ \varpi _{F}^{- \nu _{E} ( \beta)} \beta ^{e(E/F)} ].  \end{align*} We can take $\varpi _{E}=\beta ^{n_{1}} \varpi _{F}^{n_{2}} \in \mathfrak{K}( \mathfrak{A})$. Since $1=\nu _{E}(\varpi _{E})=\frac{e(E/F)}{[E:F]}\nu _{F}( {\rm det}(\beta ^{n_{1}} \varpi _{F}^{n_{2}}))$, $\nu _{F}( {\rm det} (\beta ^{n_{1}} \varpi _{F} ^{n_{2}}))>0$ and $\varpi _{E} \in \mathfrak{A}$. Similarly $ \nu _{E} (\varpi _{F}^{ - \nu _{E}( \beta )} \beta ^{e(E/F)})=0$ so $\varpi _{F} ^{ -\nu _{E} ( \beta )} \beta ^{e(E/F)} \in \mathfrak{A}$ and $\mathcal{O}_{E} \subseteq  \mathfrak{A}$.

To sum up we proved $ \mathcal{O}_{E} \subseteq \mathfrak{A}$. Therefore $ \mathcal{O}_{E} ^{\times } \subseteq U_{\mathfrak{A}} \subseteq \mathfrak{K} ( \mathfrak{A})$.     \end{proof}

\begin{remark}Assume $[\mathfrak{A},n,n-1, \beta ]$ is not equivalent to a stratum $[\mathfrak{A},n,n-1, \beta ']$ with a scalar matrix $\beta '$. By Lemma \ref{normalizercontainsE}, \cite[1.5.6]{BK} and \cite[1.4.15]{BK} our definition of a simple stratum coincides with the standard definition of a simple stratum in which the hereditary order is $\mathfrak{M}$ or $\mathfrak{I}$ and $r=n-1$ (see \cite[1.5.5]{BK}). \end{remark}

\begin{lemma}\label{lemmanue}Let $\mathfrak{A}=\mathfrak{M}$ or $\mathfrak{I}$. Let $\beta\in \mathfrak{A}$ be such that $\beta\in \mathfrak{K}(\mathfrak{A})$ and $E=F[\beta ]$ is a field. Assume $[\mathfrak{A},n,n-1, \beta ]$ is not equivalent to $[\mathfrak{A},n,n-1, \beta ']$ with a scalar matrix $\beta'$ and assume $E^{ \times } \subseteq \mathfrak{K}( \mathfrak{A})$. Then \begin{itemize} \item $e(E/F)=e(\mathfrak{A})$
\item $\nu _{E}( \beta )= \nu _{ \mathfrak{A}} (\beta)$.\end{itemize} \end{lemma}
\begin{proof}

For the first equality observe that by \cite[1.2.4 Proposition]{BK} $e(E/F)$ divides $e(\mathfrak{A})$. Pick an $\mathcal{O}_{F}$-lattice chain $\mathcal{L}=\lbrace L_{j}: \small j \in \mathbb{Z} \rbrace $ such that $\mathfrak{A}=\mathfrak{A}(\mathcal{L})$. Let $i$ be a natural number. $L_{i}/L_{i+1}$ is a vector space over $k_{F}$. Define $f( \mathfrak{A})$ to be the dimension of $L_{i}/L_{i+1}$ over $k_{F}$. Since $\mathfrak{A}$ is a principal order, $f(\mathfrak{A})$ does not depend on the choice of $i$ and one has  $e(\mathfrak{A})f(\mathfrak{A})=p$. We also have $e(E/F)f(E/F)=p$. By \cite[1.2.1 Proposition]{BK}, $L_{i}/L_{i+1}$ is a vector space over $k_{E}$ so $f(E/F)$ divides $f( \mathfrak{A})$. Therefore, $e(\mathfrak{A})=e(E/F)$.  

By the definition of $\nu_{\mathfrak{A}}$, $\beta \in \mathfrak{P}_{\mathfrak{A}}^{\nu_{\mathfrak{A}}(\beta)}\setminus \mathfrak{P}_{\mathfrak{A}}^{\nu_{\mathfrak{A}}(\beta)+1}$. Since $\beta $ is an element of the normalizer $\mathfrak{K}(\mathfrak{A})=\langle \Pi _{\mathfrak{A}}\rangle\ltimes U_{\mathfrak{A}}$ the matrix $\beta $ is of the form $\beta=\Pi _{\mathfrak{A}} ^{\nu_{\mathfrak{A}}(\beta)}C$ where $C$ is an element of $U_{\mathfrak{A}}$. 
Therefore $\nu _{F}({\rm det}(\beta))=\frac{\nu_{\mathfrak{A}}(\beta)p}{ e(\mathfrak{A})}$ and $\nu _{E}(\beta)= \frac{e(E/F)}{[E:F]} \nu _{F}({\rm det}(\beta))=\nu_{\mathfrak{A}}(\beta)=\nu _{\mathfrak{A}}(\beta)$.  
 \end{proof}       


\begin{proposition}\label{Propositionsimplestratum} 
Let $[\mathfrak{A},n,n-1, \beta ]$ be a stratum with $\mathfrak{A}= \mathfrak{M}$ or $\mathfrak{I}$ which is not equivalent to a stratum $[\mathfrak{A},n,n-1, \beta ']$ with a scalar matrix $\beta '$. The stratum $[\mathfrak{A},n, n-1, \beta ] $ is simple if and only if $n=-\nu _{\mathfrak{A}}(\beta)$ and one of the following holds
\begin{enumerate}
\item $\mathfrak{A}=\mathfrak{M}$ and the characteristic polynomial of $\varpi_{F} ^{n}\beta$ is irreducible modulo $\mathfrak{p}_{F}$ or
\item $\mathfrak{A}=\mathfrak{I} $ and $\varpi_{F} ^{\lfloor \frac{n}{p}\rfloor +1}\beta$ is of the form $\Pi _{\mathfrak{I}}^{j}B$ where $1\leqslant j\leqslant p-1$, $B\in U_{\mathfrak{I}}$. \end{enumerate}\end{proposition}
\begin{proof}
Assume that $[\mathfrak{M},n,n-1,\beta ]$ is a simple stratum which is not equivalent to a stratum $[\mathfrak{M},n,n-1,\beta']$ with a scalar matrix $\beta'$. We want to prove that the characteristic polynomial of $\varpi_{F} ^{n}\beta $ is irreducible modulo $\mathfrak{p}_{F}$ and $\nu _{\mathfrak{M}}(\beta )=-n$. The second assertion follows from the definition of a simple stratum. Since $\beta$ is minimal over $F$, $\varpi_{F} ^{-\nu _{E}(\beta )}\beta ^{e(E/F)}+\mathfrak{p}_{E}$ generates the extension $k_{E}/k_{F}$. By Lemma \ref{normalizercontainsE} and Lemma \ref{lemmanue}, $-\nu _{E}(\beta)=-\nu_{\mathfrak{M}}(\beta)=n$ and $e(E/F)=e(\mathfrak{M})=1$. Therefore $\varpi_{F} ^{n}\beta +\mathfrak{p}_{E}$ generates $k_{E}/k_{F}$. This means that the minimal polynomial of $\varpi_{F}^{n}\beta $ is irreducible modulo $\mathfrak{p}_{F}$ and is of degree $p$ so it coincides with the characteristic polynomial of $\varpi_{F}^{n}\beta$. In particular, the characteristic polynomial of $\varpi_{F} ^{n}\beta $ is irreducible modulo $\mathfrak{p}_{F}$.

Assume now that $[\mathfrak{I},n,n-1,\beta ]$ is a simple stratum which is not equivalent to a stratum $[\mathfrak{I},n,n-1,\beta']$ with a scalar matrix $\beta'$. We want to show that $n= -\nu _{\mathfrak{I}}( \beta )$ and $\varpi _{F} ^{\lfloor \frac{n}{p} \rfloor +1} \beta $ is of the form $\Pi _{\mathfrak{I}}  ^{j} B$ where $0<j<p$ and $B \in U_{\mathfrak{I}}$. By the definition of a simple stratum, $n=- \nu _{\mathfrak{I}}(\beta )$. By Lemma \ref{normalizercontainsE} and Lemma \ref{normalizer}, $\beta \in \mathfrak{K}(\mathfrak{I})=\langle\Pi _{\mathfrak{I}} \rangle\ltimes U_{\mathfrak{I}}$ and there exists a unique $j\in \mathbb{N}$ and $B\in U_{\mathfrak{I}}$ such that $ \varpi _{F}^{\lfloor \frac{n}{p} \rfloor +1}\beta=\Pi _{\mathfrak{I}} ^{j}B$. We want to show that $0<j<p$. We have \begin{equation*}j=\nu _{ \mathfrak{I}}( \varpi _{F}^{ \lfloor \frac{n}{p} \rfloor +1} \beta )=-n+p( \lfloor \frac{n}{p} \rfloor +1). \end{equation*} The element $\beta $ is minimal over $F$ so $n=- \nu _{E} ( \beta )$ is coprime with $e(E/F)=e(\mathfrak{I})=p$. Therefore $0<j=p( \lfloor \frac{n}{p} \rfloor +1)-n <p$.  
 
For the opposite direction take a stratum $[\mathfrak{M},n,n-1,\beta ]$ which is not equivalent to a stratum $[\mathfrak{M},n,n-1,\beta']$ with a scalar matrix $\beta'$ and assume that the characteristic polynomial of $\varpi_{F} ^{n}\beta $ is irreducible modulo $\mathfrak{p}_{F}$ and $\nu _{\mathfrak{M}}(\beta )=-n$. We want to show that the stratum $[ \mathfrak{M}, n,n-1, \beta ]$ is simple. $E$ is a field because the characteristic polynomial of $ \varpi_{F} ^{n} \beta $ is irreducible. We show that $\beta \in \mathfrak{K}( \mathfrak{M})$. By Lemma \ref{normalizer}, $\mathfrak{K}(\mathfrak{M})=\GL_{p}(\mathcal{O}_{F}) \rtimes \langle \varpi _{F} {\rm Id}_{p \times p} \rangle $. Denote the characteristic polynomial of $\varpi_{F} ^{n}\beta $ by $f$. Since $f$ is irreducible modulo $\mathfrak{p}_{F}$ the element $f(0)={\rm det} ( \varpi_{F} ^{n}\beta )$ does not belong to $\mathfrak{p}_{F}$. Since $\nu _{\mathfrak{M}}(\beta)=-n$, $\beta\in \mathfrak{P}_{\mathfrak{M}}^{-n}$ and $\varpi_{F} ^{n}\beta\in \mathfrak{M}$. Therefore $\varpi_{F} ^{n}\beta\in \GL_{p}(\mathcal{O}_{F})$, $\beta\in \mathfrak{K}(\mathfrak{M})$ and $ \beta \mathfrak{M} = \mathfrak{P}_{ \mathfrak{M}}^{ -n}$. It remains to show that $\beta $ is minimal over $F$. The element $\varpi _{F}^{n} \beta + \mathfrak{p}_{E} $ generates the extension of the residues fields and the extension is of degree $p$. Therefore $f(E/F) =p$, $e(E/F)=1$ and the first condition from the definition of a minimal element is satisfied. Compute $\nu _{E}( \beta )= \frac{e(E/F)}{[E:F]} \nu _{F}( {\rm det} (\beta))=-n$. Since the characteristic polynomial of $\varpi _{F}^{n} \beta = \varpi _{F}^{- \nu _{E}( \beta)} \beta ^{e(E/F)}$ is irreducible modulo $\mathfrak{p}_{F}$, $\varpi _{F}^{- \nu _{E}( \beta )} \beta ^{e(E/F)} + \mathfrak{p}_{E}$ generates the field extension  $k_{E}/k_{F}$.            

Finally, consider a stratum $[\mathfrak{I},n,n-1,\beta]$ which is not equivalent to a stratum $[\mathfrak{A},n,n-1,\beta']$ with a scalar matrix $\beta'$ and such that $\beta$ of the form $\varpi_{F} ^{-\lfloor\frac{n}{p}\rfloor -1}\Pi _{\mathfrak{I}} ^{j}B$ where $0<j<p$, $B\in U_{\mathfrak{I}}$ and $n=-\nu_{\mathfrak{I}}(\beta)$. We want to prove that the stratum $[\mathfrak{I}, n, n-1, \beta ]$ is simple. First we prove that $E=F[ \beta ] $ is a field. Let $f$ be the characteristic polynomial of $\Pi _{\mathfrak{I}} ^{j}B$. If $j=1$, then $f(x)=x^{p}$ modulo $\mathfrak{p}_{F}$ and $f(0)= {\rm det}(\Pi _{\mathfrak{I}}  B)=u \varpi _{F}$ for some $u \in \mathcal{O}_{F}^{\times}$. We deduce that $f(x)$ is Eisenstein and therefore it is irreducible. In particular, $E$ is a field. Consider now the general case $0<j<p$. Since $j$ is coprime with $p$, there exists $m_{1}, m_{2} \in \mathbb{Z}$ such that $m_{1}j +m_{2}p=1$ and $\varpi_{F}^{m_{2}}(\Pi _{ \mathfrak{I}} ^{j}B)^{m_{1}}=\Pi _{ \mathfrak{I}}  B_{1}$ for some $B_{1} \in U_{\mathfrak{I}}$. Since $\Pi _{ \mathfrak{I}}   B_{1}$ generates a field extension of degree $p$ this means that $E=F[\Pi _{ \mathfrak{I}}  ^{j}B]$ is a field. We have $\beta \in \mathfrak{K}( \mathfrak{I})$, $\nu _{\mathfrak{I}}( \beta )=-n$ and $\beta \mathfrak{I} = \mathfrak{P}_{\mathfrak{I}}^{-n}$. The characteristic polynomial $f(x)$ of $ \varpi _{F}^{ \lfloor \frac{n}{p} \rfloor +1} \beta $ is equal to $x^{p}$ modulo $\mathfrak{p}_{F}$. Therefore the extension $k_{E}/k_{F}$ is trivial and to check that $ \beta $ is minimal it is enough to check that $\nu _{E}( \beta )$ is coprime with $e(E/F)=p$. By the assumption, $\nu _{E} ( \beta )= \frac{e(E/F)}{[E:F]} \nu _{F}(  {\rm det} (\beta)) =-p( \lfloor \frac{n}{p} \rfloor +1)+j= \nu _{ \mathfrak{I}} ( \beta ) =-n $. If $p$ would divide $n$, then $\nu _{E}( \beta )=-n-p+j=-n$ and $j=p$ which is impossible. Therefore $ \beta $ is minimal over $F$.             
\end{proof}
    
\end{subsection}\begin{subsection}{Cuspidal representations of $G$}\label{sectionclasscusp}
In this subsection we recall the classification of irreducible cuspidal representations of $G=\GL_{p}(F)$. The classification originates in Carayol's work (\cite{C}). We will follow \cite{K} and \cite{BKnotes}. \cite{BKnotes} are unpublished notes so for the convenience of the reader we reproduce the proof.  

\begin{theorem}[\cite{BKnotes},\cite{K},\cite{C}]\label{theoremclassificationcusrep} Let $\pi$ be an irreducible cuspidal representation of $G$. Then there exists a character $\chi $ of $F^{\times }$ such that $\chi \pi$ is of one of the following forms: 
\begin{enumerate}
\item $\cInd _{KZ}^{G} \Lambda $ with $\Lambda$ such that $\Lambda \mid _{K}$ is inflated from some irreducible cuspidal representation of $\GL _{p}(k_{F})$,   
\item $l(\pi )>0$ and $\pi  $ contains a simple stratum $[\mathfrak{A},n,n-1,\beta ]$ with $n\geqslant 1$ and $\mathfrak{A}$ principal such that $\pi \cong \cInd _{J}^{G}\Lambda $ where $J=F[\beta ]^{\times }U_{\mathfrak{A}}^{\lfloor\frac{n+1}{2}\rfloor }$ and $\Lambda$ restricted to $U_{\mathfrak{A}}^{\lfloor \frac{n}{2}\rfloor +1}$ contains $\psi _{\beta }$.
\end{enumerate}
Moreover every representation $\pi $ such that $l(\pi)\leqslant l(\chi \pi)$ for every character $\chi$ of $F^{\times}$ satisfying $(1)$ or $(2)$ is cuspidal. \end{theorem}

\begin{remark}{(see proof of Theorem \ref{theoremclassificationcusrep})}\label{remarkminimallevel} If $\pi$ is an irreducible cuspidal representation with the minimal normalized level among all its one-dimensional twists $\pi\otimes\chi$, then $\pi$ satisfies 1. or 2. from Theorem \ref{theoremclassificationcusrep}. \end{remark}

\begin{remark}\label{remarkconjugation}If an irreducible smooth representation of $G$ contains some stratum, then it contains all strata $G$-conjugate to it. Therefore in Theorem \ref{theoremclassificationcusrep} we can assume that $\mathfrak{A}=\mathfrak{M}$ or $\mathfrak{I}$.\end{remark} 

Before the proof of Theorem \ref{theoremclassificationcusrep} we state some lemmas. 

\begin{lemma}{(see \cite[2.4.11]{BK})} \label{determinantscalar} 
Let $\mathfrak{A}$ be a principal hereditary order and let $n\in \mathbb{N}$, $n \geqslant 1$. Then any character of $U _{ \mathfrak{A}}^{n}$ which is trivial on $U_{\mathfrak{A}}^{n+1}$ and factors through a determinant is of the form $\psi _{ \beta }$ where $\beta $ is a scalar matrix. 
\end{lemma}  

\begin{lemma}\label{minimalleveltwisting} Let $\mathfrak{A}$ be a principal hereditary order. Let $\pi $ be an irreducible cuspidal representation of $\GL_{p}(F)$ which contains a simple stratum $[\mathfrak{A},n,n-1, \beta _{1} ]$. Then the following conditions are equivalent:
\begin{enumerate} 
\item there exists a stratum $[\mathfrak{A},n,n-1, \beta ]$ equivalent to $[\mathfrak{A},n, n-1, \beta  _{1}]$ such that $[F [ \beta ]:F]=1 $
\item there exists a character $\chi $ of $F^{\times}$ such that $l(\chi \pi )< l( \pi)$.
\end{enumerate} \end{lemma}


\begin{proof}[Proof of Lemma \ref{minimalleveltwisting}] First we assume that there exists a stratum $[\mathfrak{A}, n, n-1,\beta ]$ equivalent to $[\mathfrak{A}, n, n-1, \beta _{1}]$ with $[F [ \beta ]: F]=1$. We want to show that there exists a character $\chi $ of $F^{ \times }$ such that $l( \chi \pi )< l( \pi)$. Let $b\in F$ be such that $\beta = b {\rm Id}_{p\times p} $. By the assumption, $\beta \mathfrak{A}=\beta _1\mathfrak{A}=\mathfrak{P}_ \mathfrak{A}^{-n}$ so $e( \mathfrak{A})$ divides $n$. We define a character $\chi _{1}$ of $(1+ \mathfrak{p}_{F}^{ \frac{n}{e( \mathfrak{A})}})/(1+ \mathfrak{p}_{F}^{\frac{n}{e( \mathfrak{A})}+1})$: \begin{equation*}\chi _{1}( 1+z):= \psi (b  z),\quad\quad  \ \textrm{for}\quad    z\in \mathfrak{p}_{F}^{\frac{n}{e(\mathfrak{A})}}. \end{equation*}The determinant map induces the homomorphism: 
\begin{equation} \label{det}  U_{\mathfrak{A}}^{n}/ U_{\mathfrak{A}}^{n+1}\rightarrow (1+ \mathfrak{p}_{F}^{\frac{n}{e(\mathfrak{A})}})/(1+\mathfrak{p}_{F}^{\frac{n}{e( \mathfrak{A})}+1}).  \end{equation}  
Now we will show that $\chi _{1} \circ {\rm det}$ coincides with a character $\psi _{\beta} $ on $U_{\mathfrak{A}}^{n}/U_{\mathfrak{A}}^{n+1}$. For this see both $\psi _{ \beta }$ and $\chi _{1} \circ {\rm det}$ as characters of $U_{ \mathfrak{A}}^{n}$. Let 
$x\in \mathfrak{P}_{ \mathfrak{A}}^{n}$. By Leibniz formula ${\rm det} (1+x)=1+ \tr x+ y$ for some $y \in \mathfrak{p}_{F}^{2\frac{n}{e(\mathfrak{A})}}$. We have $b y \in \mathfrak{p}_{F}$ so 
\begin{equation*} \chi_{1} \circ {\rm det} (1+x)=\chi _{1}(1+\tr x +y)=\psi (b  (\tr x +y))= \psi( \tr (\beta x))= \psi _{\beta } (1+x). \end{equation*}Let $\chi _{2}$ be an extension of $\chi _{1}$ to $F^{ \times }$. Define $\chi (1+x):=\chi_{2}(1+x)^{-1}$. Then $\chi $ is a character which satisfies the desired property. 

For the converse assume that there exists a character $\chi $ of $F^{\times }$ such that $l(\chi \pi )<l(\pi )$. We look for a stratum $[\mathfrak{A},n,n-1, \beta ]$ equivalent to $[\mathfrak{A},n,n-1, \beta _{1}]$ such that $\beta$ is a scalar matrix. Denote $ \pi _{1}:= \chi \pi$. Denote by $\chi^{-1}$ the character of $F^{\times}$ such that $\chi ^{-1}(x)=\chi(x)^{-1}$ for every $x \in F^{\times}$.

The representation $\pi _{1} $ is irreducible and cuspidal. By Proposition \ref{7.13}, if $l( \pi _{1})>0$ then $\pi _{1} $ contains a simple stratum $[\mathfrak{A}_{1}, n_{1}, n_{1}-1, \gamma ]$ with $\mathfrak{A}_{1}$ principal. By the assumption, $l( \pi _{1})<l(\chi^{-1} \pi _{1})$ so $\chi ^{-1}\circ {\rm det} \otimes \psi _{\gamma }\mid _{U_{\mathfrak{A}_{1}}^{n_{1}+1}} \neq 1$. If $l(\pi _{1})=0$ then $\pi _{1}$ contains the trivial character of $\mathfrak{A}_{1}^{n_{1}+1}$ with $\mathfrak{A}_{1}= \mathfrak{M}$ and $n_{1}=0$. Therefore in both cases there exists $m \geqslant n_{1}+1$ such that \begin{equation} \label{chi} \chi^{-1} \circ {\rm det} \mid _{U_{\mathfrak{A}_{1}}^{m}} \neq 1 \quad \text{and} \quad \chi^{-1} \circ {\rm det} \mid _{U_{ \mathfrak{A}_{1}}^{m+1}} =1.  \end{equation} We can write $\chi ^{-1}\circ {\rm det} \mid _{U_{ \mathfrak{A}_{1}}^{m}} $ as $\psi _{\beta _{2}}$ for some $\beta _{2}\in \mathfrak{P}_{\mathfrak{A}_{1}}^{-m}/\mathfrak{P}_{\mathfrak{A}_{1}}^{-m+1}$. By Lemma \ref{determinantscalar} we can take $\beta _{2}$ to be a scalar matrix. 



To sum up we have proven that $\pi = \chi ^{-1} \pi _{1}$ contains a stratum $[\mathfrak{A}_{1}, m, m-1, \beta _{2} ]$. It is a fundamental stratum. By Lemma \ref{fundamentalintertwine}, the stratum $[ \mathfrak{A} _{1}, m, m-1, \beta_{2}]$ intertwines with $[ \mathfrak{A}, n, n-1, \beta _{1}]$. By \cite[2.6.1, 2.6.4]{BK} the stratum $[ \mathfrak{A}, n ,n-1, \beta _{1}]$ is equivalent to $[ \mathfrak{A} _{1}, m, m-1, \beta_{2}]$\footnote{$[\mathfrak{A}_{1},m,m-1,\beta_{2}]$ is a simple stratum in a sense of \cite{BK}.} . By the proof of \cite[1.5.2. Proposition]{BK}, $\mathfrak{A}_{1}=\mathfrak{A}$ and $m=n$. 

\end{proof}   

\begin{lemma}[see Corollary 7.15 and Theorem 9.3, \cite{BKnotes}]\label{lemmasimpleimpliescusp}Let $\pi $ be an irreducible smooth representation of $G$ which contains a simple stratum $[\mathfrak{A},n,n-1,\beta ]$ with $n\geqslant 1$ and $\mathfrak{A}$ principal. Assume that for any character $\chi$ of $F^{\times}$ we have $l(\pi)\leqslant l(\chi \pi)$. Then $\pi$ is cuspidal and there exists a simple stratum $[\mathfrak{A},n,n-1,\beta ']$ equivalent to $[\mathfrak{A},n,n-1,\beta]$ such that $\pi\cong \cInd_{J}^{G}\Lambda_{0}$ where $J=F[\beta ']^{\times }U_{\mathfrak{A}}^{\lfloor \frac{n+1}{2}\rfloor }$ and $\Lambda_{0}\mid _{U_{\mathfrak{A}}^{\lfloor\frac{n}{2}\rfloor+1}}$ contains $\psi _{\beta '}$. \end{lemma}
\begin{proof}

By the assumption, $\pi $ contains the character $\psi _{\beta }$ of $U_{ \mathfrak{A}}^{n}$. There exists an extension $\psi _{ \beta '}$ of $\psi _{\beta }$ to $U_{\mathfrak{A}}^{\lfloor \frac{n}{2}\rfloor +1}$ which is also contained in $\pi $. We have $\beta ' \equiv \: \beta \: \left({\rm mod }\: \mathfrak{P}_{\mathfrak{A}}^{1-n} \right)$. Therefore $[\mathfrak{A},n,n-1,\beta ']$ is also a smiple stratum. By \cite[1.5.8 Theorem]{BK} the $G$-intertwining of $\psi _{\beta '}$ (see \cite[1.5.7]{BK}) is $J:=F[\beta ']^{\times } U_{\mathfrak{A}}^{\lfloor \frac{n+1}{2} \rfloor }$. Since $J$ is compact modulo $Z$ there exists an irreducible smooth representation $\Lambda $ of $J$ which is contained in $\pi $ and which contains $\psi _{\beta '}$ when restricted to $U_{\mathfrak{A}}^{\lfloor \frac{n}{2} \rfloor +1}$. By \cite[Theorem 11.4 and Remark 1]{BH}, $\cInd _{J}^{G} \Lambda $ is irreducible and cuspidal. By Frobenius reciprocity $\pi \cong \cInd _{J}^{G} \Lambda $.   
 \end{proof} 

\begin{proposition}{\cite[Theorem 3.2]{K}} \label{7.13}  
Let $\pi$ be an irreducible cuspidal representation of $G$. Then $l( \pi )=0$ or $\pi $ contains a simple stratum $[ \mathfrak{A}, n,  n-1, \beta ]$ with $\mathfrak{A}$ principal. 
  \end{proposition}

Now we are ready to prove Theorem \ref{theoremclassificationcusrep}. 
\begin{proof}[Proof of Theorem \ref{theoremclassificationcusrep}] 
An irreducible smooth representation of $G$ whose normalized level is $0$ is cuspidal if and only if it is of the form $(1)$ by \cite[Theorem 8.4.1]{BK}.

Take an irreducible cuspidal representation $\pi $ of $G$ with $l(\pi )>0$. Assume that for any character $\chi $ of $F^{ \times }$ we have $l(\pi )\leqslant l( \chi \pi )$. By Proposition \ref{7.13}, $\pi $ contains a simple stratum $[\mathfrak{A},n ,n-1, \beta ]$ with $\mathfrak{A}$ principal. By Lemma \ref{minimalleveltwisting} and Lemma \ref{lemmasimpleimpliescusp}, $\pi $ is of the form as in $(2)$.

By Lemma \ref{lemmasimpleimpliescusp}, a representation $\pi $ such that $l(\pi)\leqslant l(\chi\pi)$ for every character $\chi$ of $F^{\times}$ and of the form $(2)$ is cuspidal. A one-dimensional twist of an irreducible cuspidal representation of $G$ is irreducible cuspidal.   \end{proof}

\end{subsection}

\begin{subsection}{Irreducible representations of $\GL_{p}(\mathcal{O}_{F})$ in terms of orbits}\label{orbits} 
Let $\rho $ be an irreducible smooth representation of $K=\GL_{p}(\mathcal{O}_{F})$ with conductor $r>1$. In this subsection we adjust the notation from a description of representations of $K$ as in subsection \ref{introductionorbits} to be more consistent with the notation from \cite{BH}.

Denote $l=\lfloor\frac{r+1}{2}\rfloor$ and $l'=r-l$. As in the subsection \ref{introductionorbits}, by Clifford's theorem
\begin{equation}
\label{decompmod}
\bar{\rho }\mid _{K^{l}}=m\bigoplus_{\bar{\alpha _{1}}\sim\bar{\alpha _{0}}}\bar{\varphi}_{\bar{\alpha _{1}}}
\end{equation}
for some matrix $\bar{\alpha _{0}}\in \M _{p}(\mathcal{O}_{F}/\mathfrak{p}_{F}^{l'})$, $m\in \mathbb{N}$, where $\bar{\alpha_{1}}$ runs over the conjugacy class of $\bar{\alpha _{0}}$ under $\GL_{p}(\mathcal{O}_{F}/\mathfrak{p}_{F}^{r-l})$-conjugation and $\bar{\varphi}_{\bar{\alpha_{1}}}:K^{l}\rightarrow\mathbb{C}^{\times}$ is defined as $\bar{\varphi}_{\bar{\alpha _{1}}}(1+x)=\psi(\varpi_{F} ^{-r+1}\tr (\alpha_{1}\widehat{x}))$ for some lifts $\alpha _{1}$ (resp. $\widehat{x}$) of $\bar{\alpha_{1}}$ (resp. $x$) to $ \M _{p}(\mathcal{O}_{F})$. The characters $\bar{\varphi}_{\bar{\alpha _{1}}}$ do not depend on choices of lifts. In our case it will be more convenient to look at $\rho $ as a representation of $K$ not $\GL_{p}(\mathcal{O}_{F}/\mathfrak{p}_{F}^{r})$. By (\ref{decompmod}) we can write 
\begin{equation}
\label {decomp}
\rho \mid _{U_{\mathfrak{M}}^{l}}=m\bigoplus _{\bar{\alpha _{1}}\sim \bar\alpha _{0}}\psi _{\varpi_{F} ^{-r+1}\alpha _{1}}
\end{equation}
where $\psi _{\varpi_{F} ^{-r+1}\alpha _{1}}:U_{\mathfrak{M}}^{l}\rightarrow \mathbb{C}^{ \times }$ and $\psi _{\varpi_{F} ^{-r+1}\alpha _{1}}(1+x)=\psi (\varpi_{F} ^{-r+1}\mathrm{tr}\alpha _{1}x)$. \\
For the sake of simplicity, the characters in the decomposition (\ref{decomp}) are indexed by matrices from $\M _{p}(\mathcal{O}_{F})$ instead of matrices in $\M _{p}(\mathcal{O}_{F}/\mathfrak{p}_{F}^{l'})$ as in (\ref{decompmod}) but we are still taking sums over the conjugacy class in $\M _{p}(\mathcal{O}/\mathfrak{p}_{F}^{l'})$.   
\\
We say that a representation $\rho $ contains a matrix $\alpha _{1}$ in its orbit if it admits the decomposition of the form (\ref{decomp}). We say that two orbits $\lbrace \alpha_{i}\rbrace _{i\in I}$ and $\lbrace \beta _{i}\rbrace  _{i\in J}$ are equivalent if $\lbrace \bar{\alpha} _{i}\rbrace _{i\in I}=\lbrace\bar{\beta} _{i}\rbrace _{i\in J}$ where $\bar{a}$ denotes the class of $a\in \M _{p}(\mathcal{O}_{F})$ in $\M _{p}(\mathcal{O}_{F}/\mathfrak{p}_{F}^{l'})$. Note that the notion of equivalence implicitly depends on $r$. From now on we consider the orbits up to equivalence.    
\\
\begin{remark} \label{rho}
By Clifford theory, if a representation $\rho $ admits the decomposition (\ref{decompmod}) then
\begin{equation*} 
\bar{\rho} \cong \mathrm{Ind} _{\mathrm{Stab}_{\GL_{p}(\mathcal{O}_{F}/\mathfrak{p}_{F}^{r})}\bar{\varphi } _{\bar{\alpha } _{1} }}^{\GL_{p}(\mathcal{O}_{F}/\mathfrak{p}_{F}^{r})}\bar{\theta } 
\end{equation*}
where $\bar{\theta }$ is an irreducible representation of $\mathrm{Stab}_{\GL_{p}(\mathcal{O}_{F}/\mathfrak{p}_{F}^{r})}\bar{\varphi } _{\bar{\alpha } _{1}}$ which contains $\bar{\varphi } _{\bar{\alpha } _{1}}$. Therefore as a representation of $K$, $\rho $ is isomorphic to $\mathrm{Ind}_{\mathrm{Stab}_{K}\bar{\varphi} _{\bar{\alpha} _{1}}}^{K} \theta $ where $\theta $ is an inflation of $\bar{\theta }$ to $\mathrm{Stab}_{K}\bar{\varphi} _{\bar{\alpha} _{1}}$.\end{remark}   
\end{subsection}
\end{section}

\begin{section}{Cuspidal types on $K$ in terms of orbits} 
\label{sectionproof} 
In this section we give a description of orbits of cuspidal types. We also determine which orbits provide cuspidal types under the condition that the conductor of a cuspidal type is at least $4$. This in particular allows us to determine which cuspidal types on $K$ with conductor at least $4$ are regular representations.

\begin{subsection}{Cuspidal types on $\GL_{p}( \mathcal{O} _{F})$} The goal of this subsection is to prove the following theorem.
\begin{theorem} \label{TheoremGLp}
If $\lambda $ is a cuspidal type on $K=\GL_{p}( \mathcal{O}_{F})$, then it is a one-dimensional twist of one of the following:
\begin{enumerate} \item a representation which is inflated from an irreducible cuspidal representation of $\GL_{p}(k_{F})$;
\item a representation whose orbit contains a matrix whose characteristic polynomial is irreducible modulo $ \mathfrak{p}_{F}$; 
\item a representation whose orbit contains a matrix of the form $\Pi _{ \mathfrak{I}}^{j} B$ where $0<j<p$ and $B \in U_{ \mathfrak{I}}$.\end{enumerate}
Moreover, if $\pi$ is an irreducible cuspidal representation of $G$ such that $l(\pi)\leqslant l(\chi\pi)$ for all one-dimensional characters $\chi:F^{\times }\to \mathbb{C}^{\times}$, then the cuspidal type for $\mathfrak{I}(\pi)$ is of the form $(1)$, $(2)$ or $(3)$ (no twisting is needed).
Conversely if a representation is a one-dimensional twist of a representation of the form $(3)$ and has conductor at least $4$, or is of the form $(1)$ or $(2)$, then it is a cuspidal type. \end{theorem} 

\begin{corollary}[{see section \ref{regularity}}] Let $\pi $ be an irreducible cuspidal representation of $G$ such that $l(\pi)\leqslant l(\chi\pi)$ for all one-dimensional characters $\chi:F^{\times}\to\mathbb{C}^{\times}$. The cuspidal type for $\mathfrak{I}(\pi )$ is regular if and only if $l(\pi )=m$ or $m-\frac{1}{p}$ for certain $m \in \mathbb{N}\setminus\{0\}$.  
\end{corollary}

Before the proof of Theorem \ref{TheoremGLp} we state auxiliary lemmas.  
\begin{lemma}
\label{lemma1}
Let $\mathfrak{A}=\mathfrak{M}$ or $\mathfrak{I}$ and let $J$ be an open compact modulo $Z$ subgroup of $\mathfrak{K}(\mathfrak{A})$. Denote by $J^{\circ}$ the maximal compact subgroup of $J$. Then $J^{\circ}=J\cap U_{\mathfrak{A}}=J\cap K$. 
\end{lemma}
\begin{proof}
By Lemma \ref{normalizer}, $\mathfrak{K}(\mathfrak{A})=\langle\Pi _{ \mathfrak{A}} \rangle\ltimes U_{\mathfrak{A}}$. Since any non-trivial subgroup of $\langle\Pi _{ \mathfrak{A}} \rangle$ is not compact,
 $J^{\circ}$ has to be contained in $U_{\mathfrak{A}}$ and $J^{\circ}\subseteq J\cap U_{\mathfrak{A}}$. The subgroup $J$ is closed and $U_{ \mathfrak{A}}$ is compact so $J \cap U_{ \mathfrak{A}}$ is compact. Therefore $J^{ \circ} = J \cap U_{ \mathfrak{A}}$. Since $J^{\circ}$ is the unique maximal compact subgroup of $J$ and $J\cap K$ is compact, $J \cap K \subseteq J^{\circ }$. We also have $U_{ \mathfrak{A}} \subseteq K$ and $J^{ \circ } =J \cap U_{ \mathfrak{A}} \subseteq J \cap K$. 
\end{proof} 
Let $H$ be a locally profinite group. Let $\pi _{1} $ and $\pi _{2}$ be representations of $H$. We write $\pi _{1} \sim \pi _{2}$ if there exists $h \in H$ such that $\pi _{1} = \pi _{2}^{h}$. The following is a variation on Clifford's theorem: 

\begin{proposition} \label{Clifford2} Let $H$ be a locally profinite group, $N$ a normal open compact subgroup of $H$ and let $(\pi _{1}, V_{1})$ be an irreducible admissible smooth representation of $H$. Then \begin{equation*} \pi _{1} \mid _{N} = m\bigoplus _{\rho _{1} \sim \rho } \rho _{1} \end{equation*} for certain $m \in \mathbb{Z}$ and some irreducible smooth representation $\rho$ of $N$. 
\end{proposition} 
\begin{proof} Denote by $\widehat{N}$ the set of equivalence classes of irreducible smooth representations of $N$. Let $\rho \in \widehat{N}$. The $\rho$-isotypic component of $V_{1}$ is the sum of all irreducible $N$-subspaces of $V_{1}$ of class $\rho $. We denote it by $V_{1}^{\rho }$. By \cite[2.3 Proposition]{BH}, 
\begin{equation*} V_{1}= \bigoplus _{ \rho \in \widehat{N}}V _{1}^{\rho }. \end{equation*} 
Fix some $\rho \in \widehat{N}$. Since $\pi _{1}$ is irreducible we have $V_{1}=\sum _{ g \in G} g V _{1}^{ \rho }= \sum _{ g \in G} V _{1}^{\rho ^{g}}= \bigoplus _{ \rho _{1} \sim \rho }V _{1}^{ \rho _{1}}$. Since $\pi _{1}$ is admissible and $\textrm{ker}(\rho )$ is open $V _{1}^{\rho} \subseteq V_{1}^{ \textrm{ker} ( \rho )}$ is finite dimensional hence $V _{1}^{ \rho } \cong m \rho$ for certain $m \in \mathbb{Z}$. Therefore $\pi _{1} \mid _{N} =m \bigoplus _{\rho _{1} \sim \rho } \rho _{1}$.  
\end{proof} 
\begin{remark} 
\label{Cliffordremark}   
If $H \subseteq \GL _{p}(F)$ is compact modulo $Z$ then by \cite[Theorem 2.1]{Mosk} every irreducible smooth representation of $H$ is finite dimensional and hence admissible.  
\end{remark}

\begin{lemma} 
\label{extensionofcuspidaltype} 
Let $U$ be a compact open subgroup of $K$ and let $\pi $ be an irreducible cuspidal representation of $G$. Let $\rho ' $ be a cuspidal type on $U$ for $\mathfrak{I}( \pi )$ and let $\rho $ be an irreducible smooth representation of $K$ which contains $\rho ' $. Assume that $\rho $ is contained in $\pi \mid _{K}$. Then $\rho $ is a cuspidal type on $K$ for $\mathfrak{I}(  \pi )$. 
\end{lemma} 
\begin{proof} 
Take an irreducible smooth representation $\pi _{1}$ of $G$. We want to show that $\pi _{1}\mid _{K}$ contains $\rho $ if and only if $\mathfrak{I}(\pi _{1} )= \mathfrak{I}( \pi )$. If $\pi _{1} \mid _{K}$ contains $\rho $ then it also contains $\rho '$ and by the assumption $\mathfrak{I}( \pi )= \mathfrak{I}( \pi _{1})$. For the reverse implication assume that $\mathfrak{I}( \pi _{1})=\mathfrak{I}( \pi )$. Since $\pi _{1} \cong \pi \otimes \chi \circ {\rm det}$ for a certain unramified character $\chi $ of $F^{\times }$, $\pi _{1} \mid _{K} \cong \pi \mid _{K} $. Therefore $\pi _{1}$ contains $\rho $.    
\end{proof}        

\begin{lemma}\label{Lambdarestrictedirreducible} Consider a stratum $[ \mathfrak{A},n, n-1, \alpha ]$ with $\mathfrak{A}$ principal, $n \geqslant 1$. Denote $J= F[ \alpha ]^{ \times } U_{ \mathfrak{A}}^{ \lfloor \frac{n+1}{2} \rfloor}$ and let $ \pi $ be an irreducible cuspidal representation such that $\pi \cong \cInd _{J}^{G} \Lambda $ for some irreducible smooth representation $\Lambda$ of $J$. Denote by $J^{ \circ }$ the maximal compact subgroup of $J$. Then $\Lambda \mid _{J^{ \circ}}$ is irreducible.  
\end{lemma}  
\begin{proof}
Let $\Lambda \mid _{J^{ \circ }}= \sum _{i \in I} \lambda _{i}$ for a certain set $I$ and irreducible representations $\lambda _{i}$ of $J^{ \circ}$. First we show that for any $i, l \in I$ we have $\lambda _{i} \cong \lambda _{l}$. Take $i, l \in I$. By Proposition \ref{BKprop}, the representations $\lambda _{i}$ and $\lambda _{l}$ are cuspidal types on $J^{\circ }$ for $\mathfrak{I}(\pi )$. By Lemma \ref{extensionofcuspidaltype}, irreducible components of $\cInd _{J^{\circ }}^{K} \lambda _{i}$ and $\cInd _{J^{\circ }}^{K} \lambda _{l}$ are cuspidal types on $K$ for $\mathfrak{I}(\pi )$. By Theorem \ref{Paskunas}, a cuspidal type on $K$ for $\mathfrak{I}(\pi )$ is unique and appears in $\pi $ with multiplicity one so $\cInd _{J^{\circ }}^{K}\lambda _{i} $ and $\cInd _{J^{\circ }}^{K} \lambda _{l}$ are irreducible and $\cInd _{J ^{\circ }}^{K} \lambda _{i} \cong \cInd _{ J^{ \circ }}^{K} \lambda _{l}$. By Frobenius reciprocity and Mackey's formula, this implies that there exists $k \in K$ which intertwines $\lambda _{l}$ with $\lambda _{i}$, i.e. 
\begin{equation}\label{lambdaintertwine}\textrm{Hom} _{J^{ \circ } \cap (J^{ \circ })^{k}}\left(\lambda _{l}^{k} \mid _{J^{\circ } \cap (J^{ \circ})^{k}}, \lambda _{i} \mid _{J^{ \circ } \cap (J^{ \circ })^{k}}\right)\neq 0. \end{equation} 
On the other hand we can apply Proposition \ref{Clifford2} and Remark \ref{Cliffordremark} to the representation $\Lambda$. The group $J^{\circ }$ is the maximal compact subgroup of $J$ so after $J$-conjugation it remains the maximal compact subgroup of $J$. Therefore $J^{ \circ }$ is a normal subgroup of $J$. By Proposition \ref{Clifford2}, there exists $j \in J$ such that $ \lambda _{l} \cong \lambda _{i} ^{j}$. Together with (\ref{lambdaintertwine}) this implies 
\begin{equation*} \textrm{Hom}_{(J^{ \circ })^{k}\cap J^{ \circ }} \left( \lambda _{i}^{jk} \mid _{(J^{ \circ})^{k} \cap J^{ \circ }}, \lambda _{i} \mid _{(J^{ \circ })^{k} \cap J^{ \circ}} \right) \neq 0. 
\end{equation*}  Since $J^{ \circ }$ is normal in $J$ we have $(J ^{ \circ })^{jk}=(J^{ \circ })^{k}$ and $jk$ intertwines $\lambda _{i}$. In particular, $jk$ intertwines $\Lambda $. An element from $G$ intertwines $\Lambda $ with itself if and only if it belongs to $J$ as otherwise $\cInd _{J}^{G} \Lambda $ would not be irreducible (see \cite[11.4 Theorem and 11.4 Remark 1,2]{BH}). This means that $jk \in J$, so $k \in K \cap J=J ^{ \circ}$. By (\ref{lambdaintertwine}), $\lambda _{i} \cong \lambda _{l}$. 

We proved $\Lambda \mid _{J^{ \circ }}=m \lambda _{i}$ for some $m \in \mathbb{Z}$. By Mackey formula, $\pi \mid _{K}$ contains $\cInd _{J ^{ \circ }}^{K}( \Lambda \mid _{J ^{ \circ }})=m \cInd _{J^{ \circ }}^{K} \lambda _{i}$. The representation $\cInd _{J^{ \circ }}^{K} \lambda _{i}$ is a cuspidal type for $\mathfrak{I}(\pi )$ on $K$. By Theorem \ref{Paskunas}, $m=1$ so $\Lambda \mid _{ J^{ \circ }}= \lambda _{i}$ is irreducible.   

\end{proof}
By Proposition \ref{BKprop}, Lemma \ref{extensionofcuspidaltype} and Lemma \ref{Lambdarestrictedirreducible} we know the following.   

\begin{lemma}[cf. {\cite[Proposition 3.1]{P}}]\label{inductioniscuspidaltype}  
Let $\pi $ and $\Lambda $ be as in Lemma \ref{Lambdarestrictedirreducible}. Then $\cInd _{J^{\circ}}^{K}(\Lambda \mid _{J^{\circ }})$ is a cuspidal type on $K$ for $\mathfrak{I} (\pi )$.
\end{lemma}
Now we are ready to prove Theorem \ref{TheoremGLp}

\begin{proof}[Proof of Theorem \ref{TheoremGLp}]First we prove that if a representation is a cuspidal type on $K$ for $\mathfrak{I}(\pi)$ where $\pi$ is an irreducible cuspidal representation of $G$ such that $l(\pi)\leqslant l(\chi\pi)$ for every one-dimensional character $\chi:F^{\times}\to\mathbb{C}^{\times}$, then it is of the form $(1)$, $(2)$ or $(3)$ from Theorem \ref{TheoremGLp}. 

Let $\lambda$ be the cuspidal type on $K$ for $\mathfrak{I}(\pi)$ where $\pi$ is an irreducible cuspidal representation of $G$ such that $l(\pi)\leqslant l(\chi\pi)$ for every one-dimensional character $\chi:F^{\times}\to\mathbb{C}^{\times}$. By Remark \ref{remarkminimallevel}, we can restrict our considerations to the following two cases.



\textbf{Case 1}, $l(\pi )=0$. By Theorem \ref{theoremclassificationcusrep}, there exists a representation $\Lambda$ of $ZK$ which is an extension of an inflation of some irreducible cuspidal representation of $\GL_{p}(k_{F})$, such that $\pi\cong \cInd _{ZK}^{G}\Lambda$. By Proposition \ref{BKprop}, $\Lambda\mid _{K}$ is a cuspidal type on $K$ for $\mathfrak{I}(\pi) $. By Paskunas' uniqueness theorem (\cite{P}, Theorem 1.3), $\lambda\cong\Lambda\mid _{K}$. Therefore $\lambda $ is of type $(1)$ from Theorem \ref{TheoremGLp}.

\textbf{Case 2}, $l(\pi)>0$ and $\pi$ contains a simple stratum $[\mathfrak{A},n,n-1,\alpha ]$ with $n\geqslant 1$, a principal order $\mathfrak{A}$, such that $\pi\cong\cInd _{J}^{G}\Lambda $ where $J=F[\alpha ]^{\times }U_{\mathfrak{A}}^{\lfloor\frac{n+1}{2}\rfloor }$
and $\Lambda\mid _{U_{\mathfrak{A}}^{\lfloor\frac{n}{2}\rfloor +1 }}=m \psi _{\alpha}$ for some $m\in \mathbb{N}$. The subgroup $J$ is open, contains and is compact modulo $Z$. Since $\alpha $ is minimal, $J\subseteq \mathfrak{K}(\mathfrak{A})$. Recall that $J^\circ=J\cap K$.
By Lemma \ref{Lambdarestrictedirreducible}, $\Lambda\mid_{J^{\circ}}$ is irreducible and by Proposition \ref{BKprop}, $\Lambda\mid _{J^{\circ}}$ is a cuspidal type on $J^{ \circ }$ for $\mathfrak{I}(\pi) $. By Lemma \ref{inductioniscuspidaltype}, ${\rm Ind}_{J^{\circ}}^{K}(\Lambda \mid _{J^{\circ}})$ is a cuspidal type on $K$ for $\mathfrak{I}(\pi) $ and by Theorem \ref{Paskunas}, ${\rm Ind}_{J^{\circ}}^{K}(\Lambda \mid _{J^{\circ}})\cong \lambda$. By Mackey's formula, $\lambda\mid_{U_{\mathfrak{A}}^{\lfloor\frac{n}{2}\rfloor+1}}$ contains $\psi _{\alpha }$.
By Remark \ref{remarkconjugation}, we can assume $\mathfrak{A}=\mathfrak{M}$ or $\mathfrak{A}=\mathfrak{I}$.

\textbf{Case 2.1} $\mathfrak{A}=\mathfrak{M}$. The representation $\lambda$ restricted to $U_{\mathfrak{M}}^{n+1}$ contains $\psi _{\alpha }\mid _{U_{\mathfrak{M}}^{n+1}}=1_{U_{\mathfrak{M}}^{n+1}}$. Since $U_{\mathfrak{M}}^{n+1}$ is an open normal subgroup of $K$ and $K$ is compact, $\lambda\mid _{U_{\mathfrak{M}}^{n+1}}$ is a direct sum of irreducible representations and each of them is conjugated to the trivial character. This means $\lambda\mid_{U_{\mathfrak{M}}^{n+1}}$ is trivial and $\lambda $ factors through $\GL_{p}(\mathcal{O}_{F}/\mathfrak{p}_{F}^{n+1})$. Since $\nu_{\mathfrak{M}}(\alpha )=-n$,  the character $\psi _{\alpha }$ as a character of $U_{\mathfrak{M}}^{n}$ is non-trivial and $\lambda $ does not factor through $\GL_{p}(\mathcal{O}_{F}/\mathfrak{p}_{F}^{n})$. Therefore $\lambda $ has conductor $r=n+1$.

As before, let $l=\lfloor\frac{r+1}{2}\rfloor$ and let $ \alpha _{0}:=\varpi _{F}^{r-1} \alpha $. The restriction $\lambda\mid _{U_{\mathfrak{M}}^{l}}$ contains $\psi _{\alpha }$. Therefore $\alpha _{0}$ is contained in the orbit of $\lambda $. By Proposition \ref{Propositionsimplestratum}, the characteristic polynomial of $\alpha _{0}$ is irreducible mod $\mathfrak{p}_{F}$. We conclude that $\lambda$ is of type $(2)$ from Theorem \ref{TheoremGLp}.
 
\textbf{Case 2.2} $\mathfrak{A}=\mathfrak{I}$. We compute the conductor of $\lambda $ in terms of $n$. We know that $\lambda$ restricted to $U_{\mathfrak{I}}^{\lfloor\frac{n}{2}\rfloor+1}$ contains $\psi _{\alpha}$. However to compute the conductor we need information about the restrictions to subgroups $U_{\mathfrak{M}}^{j}$ for certain $j$. To switch between these two types of subgroups we will use the following inclusion:\begin{equation}\label{containing}
U_{\mathfrak{I}}^{p+1+pm}\supseteq U_{\mathfrak{M}}^{2+m}\qquad\text{for all}\quad m\in\mathbb{N}.
\end{equation}
Since $\lambda\mid_{U_\mathfrak{I}^{\lfloor\frac{n}{2}\rfloor+1}}$ contains $\psi _{\alpha}$, $\lambda \mid_{U_{\mathfrak{I}}^{n+1}}$ contains the trivial character. By (\ref{containing}), $\lambda\mid _{U_{\mathfrak{M}}^{\lfloor\frac{n}{p}\rfloor+2}}$ contains the trivial character. Similarly as in the case 2.1 we deduce that $\lambda $ factors through $\GL_{p}(\mathcal{O}_{F}/\mathfrak{p}_{F}^{\lfloor\frac{n}{p}\rfloor +2})$. 
By Proposition \ref{Propositionsimplestratum}, $\alpha=\varpi_{F}^{-\lfloor\frac{n}{p}\rfloor -1}\Pi_{\mathfrak{I}}^{j}B$ where $0<j:=p(\lfloor  \frac{n}{p} \rfloor +1)-n<p $ and $B\in U_{\mathfrak{I}}$. The restriction $\lambda \mid _{U_{ \mathfrak{I}}^{ \lfloor \frac{n}{2} \rfloor +1} \cap U _{ \mathfrak{M}}^{ \lfloor \frac{n}{p} \rfloor +1}}$ contains $ \psi _{\alpha } \mid _{U _{ \mathfrak{I}}^{ \lfloor \frac{n}{2} \rfloor +1} \cap U _{ \mathfrak{M}}^{ \lfloor \frac{n}{p} \rfloor +1}}$. Assume $ \psi _{ \alpha } \mid _{ U _{ \mathfrak{I}}^{ \lfloor \frac{n}{2} \rfloor +1} \cap U _{ \mathfrak{M}}^{ \lfloor \frac{n}{p} \rfloor +1}}$ is trivial. Then $ \alpha \in \mathfrak{P} _{ \mathfrak{I}}^{- \lfloor \frac{n}{2} \rfloor } + \mathfrak{P} _{ \mathfrak{M}}^{ - \lfloor \frac{n}{p} \rfloor}$ and  
\begin{align*} 
B=& \varpi _{F} ^{\lfloor \frac{n}{p} \rfloor  +1} \Pi _{ \mathfrak{I}} ^{-j} \alpha \in  \Pi _{\mathfrak{I}}^{-j} \mathfrak{P} _{ \mathfrak{M}} + \mathfrak{P} _{ \mathfrak{I}} ^{ - \lfloor \frac{n}{2} \rfloor +p \lfloor \frac{n}{p} \rfloor +p-j}\\=& \Pi _{ \mathfrak{I}}^{p-j} \mathfrak{M} + \mathfrak{P} _{ \mathfrak{I}} ^{ n- \lfloor \frac{n}{2} \rfloor } \subseteq \Pi _{ \mathfrak{I}} \left( \mathfrak{M} + \mathfrak{I} \right)=\Pi_\mathfrak{J}\mathfrak M.  
\end{align*} 
In particular, $\det B\in \mathfrak p_F$. Since $B \in U _{ \mathfrak{I}} $, this leads to a contradiction. Therefore $ \psi _{ \alpha } \mid _{U _{ \mathfrak{M}}^{ \lfloor \frac{n}{p} \rfloor +1} \cap U _{ \mathfrak{I}}^{ \lfloor \frac{n}{2} \rfloor +1}}$ is non-trivial. Consequently, $ \lambda \mid _{U _{ \mathfrak{M}}^{ \lfloor \frac{n}{p} \rfloor +1}}$ is non-trivial and $\lambda $ is of conductor $\lfloor\frac{n}{p}\rfloor+2$.

Write $r=\lfloor\frac{n}{p}\rfloor+2$ and $l=\lfloor \frac{r+1}{2} \rfloor $. Take $\beta_{0}$ from the orbit of $\lambda$ and let $\beta=\varpi_{F}^{-r+1}\beta_{0}$. 
Since $\lambda\mid _{U_{\mathfrak{I}}^{\lfloor\frac{n}{2}\rfloor +1}}$ contains $\psi _{\alpha}$, we know that $\lambda\mid_{U_{\mathfrak{M}}^{l}\cap U_{\mathfrak{I}}^{\lfloor\frac{n}{2}\rfloor +1}}$ contains $\psi _{\alpha }\mid _{U_{\mathfrak{M}}^{l}\cap U_{\mathfrak{I}}^{\lfloor\frac{n}{2}\rfloor +1}}$. On the other hand by the definition of the orbit of $\lambda $ 
\begin{equation}
 \lambda \mid _{U_{\mathfrak{M}}^{l}\cap U_{\mathfrak{I}}^{\lfloor\frac{n}{2}\rfloor +1}}=m_{\lambda}\bigoplus _{\bar{\beta_{0} '} \sim \bar{\beta_{0}}}\psi _{\varpi_{F}^{-r+1}\beta_{0} '}\mid _{U_{\mathfrak{M}}^{l}\cap U_{\mathfrak{I}}^{\lfloor\frac{n}{2}\rfloor +1}}\end{equation}
where $m_{ \lambda }\in \mathbb{N}$ and $\bar{\beta_{0}'}$ (resp. $\bar{\beta_{0}}$) is the image of $\beta_{0}'$ (resp. $\beta _{0}$) in $\textrm{M}_{p}(\mathcal{O}_{F}/\mathfrak{p}_{F}^{r-l})$ and $\bar{\beta_{0}'}$ runs over the conjugacy class of $\bar{\beta_{0}}$ under $\textrm{GL}_p(\mathcal{O}_{F}/\mathfrak{p}_{F}^{r-l})$. Therefore we can assume

\begin{equation}\label{eofchar}\psi_{\beta}\mid _{U_{\mathfrak{M}}^{l}\cap U_{\mathfrak{I}}^{\lfloor\frac{n}{2}\rfloor +1}}=\psi _{\alpha} \mid _{U_{\mathfrak{M}}^{l}\cap  U_{\mathfrak{I}}^{\lfloor \frac{n}{2}\rfloor +1}} . \end{equation}

The rest of the proof of this case was suggested by Shaun Stevens. The alternative version of the proof can be found the author's Ph.D. thesis \cite{ASthesis}. The equation (\ref{eofchar}) implies that $\beta - \alpha \in \mathfrak{P}_{\mathfrak{M}}^{1-l}+\mathfrak{P}_{\mathfrak{I}}^{-\lfloor \frac{n}{2} \rfloor}$. Therefore, 
\begin{equation*} 
\beta  \in \Pi _{ \mathfrak{I}}^{-n}U_{ \mathfrak{I}} + \mathfrak{P}_{\mathfrak{M}}^{-l+1} + \mathfrak{P}_{ \mathfrak{I}}^{ -\lfloor \frac{n}{2} \rfloor} \subseteq \Pi _{\mathfrak{I}}^{-n}U_{\mathfrak{I}}+ \mathfrak{P}_{ \mathfrak{M}}^{-l+1}.     
\end{equation*} 
Take $\beta_{1} \in \Pi _{ \mathfrak{I}}^{-n} U _{ \mathfrak{I}}$ such that $\beta - \beta_{1} \in \mathfrak{P}_{ \mathfrak{M}}^{-l+1}$. Since $\psi _{\beta }$ is a character of $U_{ \mathfrak{M}}^{l}$ we have $\psi _{\beta }=\psi _{ \beta_{1}}$ and we can assume that $\beta \in \Pi _{ \mathfrak{I}}^{-n}U _{ \mathfrak{I}}$. Therefore, $\lambda$ is of the form $(3)$ from Theorem \ref{TheoremGLp}. This concludes the proof of the first direction.

Now we prove that a one-dimensional twist of a representation which is of the form $(3)$ and has conductor at least $4$ or of the form $(1)$ or $(2)$ is a cuspidal type. 

\textbf{Case 3} Let $\rho $ be an irreducible smooth representation of $K$ which is inflated from an irreducible cuspidal representation of $\GL_{2}(k_{F})$. We can extend $\rho $ to an irreducible representation of $KZ$. Denote this extension by $\Lambda $. By Theorem \ref{theoremclassificationcusrep}, $\pi =\cInd _{KZ}^{G}\Lambda $ is an irreducible cuspidal representation of $G$. By Proposition \ref{BKprop}, $\rho $ is a cuspidal type on $K$ for $\mathfrak{I}(\pi)$.

\textbf{Case 4} Assume $\rho $ is an irreducible smooth representation of $K$ either whose orbit contains a matrix with characteristic polynomial irreducible mod $\mathfrak{p}_{F}$ or with conductor at least $4$ and an orbit containing a matrix of the form $\Pi _{ \mathfrak{I}}^{j}B$ for $j \in \mathbb{N}$, $0<j<p$ and $B \in U_{ \mathfrak{I}}$. Denote this matrix by $\beta_{0}$ and let $r$ be the conductor of $\rho $. Let $\pi $ be an irreducible smooth representation of $G$ which contains $\rho $. First we will prove that $\pi $ contains a simple stratum. We will divide the proof into two subcases depending on the properties of $\beta _{0}$. 

\textbf{Case 4.1} Assume $\beta _{0}$ is a matrix with the characteristic polynomial irreducible mod $\mathfrak{p}_{F}$. Let $n:=r-1$ and $\beta :=\varpi_{F} ^{-n}\beta _{0}$. Since $\pi $ contains the character $\psi _{\beta }$ of $U_{\mathfrak{M}}^{\lfloor\frac{r+1}{2}\rfloor}\supseteq U_{\mathfrak{M}}^{n}$ the representation $\pi $ contains the stratum $[\mathfrak{M},n, n-1, \beta ]$. Since the characteristic polynomial of $\beta_{0}$ is irreducible modulo $\mathfrak{p}_{F}$, this stratum is not equivalent to a stratum $[\mathfrak{M},n,n-1,\beta ']$ with a scalar matrix $\beta'$. By Proposition \ref{Propositionsimplestratum}, the stratum is simple. By Lemma \ref{lemmasimpleimpliescusp}, $\pi$ is cuspidal and there exists a simple stratum $[\mathfrak{M},n,n-1,\beta_{1}]$ equivalent to $[\mathfrak{M},n,n-1,\beta]$ such that $\pi\cong \cInd_{J}^{G}\Lambda$ where $J=F[\beta_{1}]^{\times}U_{\mathfrak{M}}^{\lfloor \frac{n+1}{2}\rfloor}$ and $\Lambda\mid_{U_{\mathfrak{M}}^{\lfloor\frac{n}{2}\rfloor+1}}$ contains $\psi _{\beta_{1}}$.
By Lemma \ref{inductioniscuspidaltype}, to finish the proof it is is enough to show $\rho \cong {\rm Ind}_{J\cap K}^{K}\left( \Lambda \mid _{J\cap K} \right)$. Since $\rho $ is contained in $\pi $, we have 
 \begin{equation*}{\rm Hom}_{K}(\rho  , (\cInd _{J}^{G}\Lambda )\mid _{K})\neq 0.
 \end{equation*} 
By Mackey formula and Frobenius reciprocity there exists $g\in J\setminus G/ K$ such that
\begin{equation}\label{rhoInd} {\rm Hom}_{K}(\rho , \cInd _{J^{g}\cap K}^{K}(\Lambda ^{g} \mid _{J^{g}\cap K}))\cong {\rm Hom}_{J^{g}\cap K}(\rho \mid _{J^{g}\cap K} , \Lambda ^{g} \mid _{J^{g}\cap K})\neq 0
\end{equation} 
 where $\Lambda ^{g}$ denotes the representation $\Lambda ^{g}(x)=\Lambda (gxg^{-1})$ for any $x\in J^{g}\cap K$. 
In particular, ${\rm Hom}_{U_{\mathfrak{M}}^{n}\cap (U_{\mathfrak{M}}^{n})^{g}}(\rho , \Lambda ^{g} )\neq 0$. Denote the subgroup $U_{\mathfrak{M}}^{n}$ by $H$. By Proposition \ref{Clifford2} and Remark \ref{Cliffordremark}, $\rho  \mid _{H}$ is a multiple of a direct sum of one-dimensional representations and each of them is conjugate to $\psi _{\beta  }$ by an element of $K$. Therefore there exist $g_{1}\in K$ such that\begin{center}
${\rm Hom}_{H\cap H^{g}}(\psi _{\beta  }^{g_{1}}\mid _{H\cap H^{g}},\psi _{\beta_{1} }^{g}\mid _{H\cap H^{g}})\neq 0$.
\end{center} Since $g_{1}$ normalizes $H$, the previous is equivalent to
\begin{center}
 ${\rm Hom}_{H^{g_{1}}\cap (H^{g_{1}})^{g_{1}^{-1}g}}(\psi _{\beta  ^{g_{1}}}\mid _{H^{g_{1}}\cap (H^{g_{1}})^{g_{1}^{-1}g}},(\psi _{\beta_{1} ^{g_{1}}})^{g_{1}^{-1}g}\mid _{H^{g_{1}}\cap (H^{g_{1}})^{g_{1}^{-1}g}})\neq 0$
 \end{center}
 where $\beta ^{g_{1}}=g_{1}^{-1}\beta  g_{1}$. Since $\beta ^{g_{1}}\equiv \beta_{1}^{g_{1}} \mod \mathfrak{P}_{\mathfrak{M}}^{1-n}$, $g_{1}^{-1}g$ intertwines the stratum $[\mathfrak{M},n, n-1, \beta_{1} ^{g_{1}} ]$. Since $[\mathfrak{M},n, n-1, \beta_{1} ]$ is simple and $g_{1}\in K$ the stratum $[\mathfrak{M},n, n-1, \beta_{1} ^{g_{1}}]$ is also simple. By \cite[1.5.8]{BK}, $g_{1}^{-1}g\in F[\beta_{1}^{g_{1}}]^{\times}K$ and therefore $g\in F[\beta_{1}]^{\times}K= J K$. By (\ref{rhoInd}), $\rho $ is isomorphic to a subrepresentation of $ \cInd _{J\cap K}^{K}(\Lambda \mid _{J\cap K})$ so by Lemma \ref{inductioniscuspidaltype} $\rho $ is a cuspidal type on $K$ for $\mathfrak{I}( \pi )$. 
 
\textbf{Case 4.2} Assume now that $\rho $ has conductor $r>3$ and has an orbit containing a matrix $\beta _{0}$ of the form $\Pi _{\mathfrak{I}}^{j}B$ for some $j \in \mathbb{N}$, $0<j<p$ and $B \in U_{ \mathfrak{I}}$. Denote $\beta = \varpi _{F}^{-r+1} \beta _{0}$. We have $\nu _{ \mathfrak{I}}( \beta )= -p (r-1) +j$. Put $n:=p(r-1)-j$. The stratum $[\mathfrak{I}, n,n-1,\beta ]$ is simple. We have 
\begin{equation*}U_{ \mathfrak{I}}^{n} = 1+ \varpi _{F}^{r-2} \mathfrak{P}_{ \mathfrak{I}}^{p-j} \subseteq U_{ \mathfrak{M}}^{r-2}.
\end{equation*}For $r > 3$ the last group is contained in $U_{ \mathfrak{M}}^{\lfloor \frac{r+1}{2} \rfloor }=U_{\mathfrak{M}}^{l}$. Since $\pi$ contains $ \psi _{ \beta }$ as a character on $U_{ \mathfrak{M}}^{l}$, it also contains its restriction to $U_{ \mathfrak{I}}^{n}$. Therefore, $\pi $ contains the simple stratum $[ \mathfrak{I}, n, n-1, \beta ]$. Since $\nu _{\mathfrak{I}}(\beta)$ and $p$ are coprime, the stratum $[\mathfrak{I},n,n-1,\beta]$ is not equivalent to a stratum $[\mathfrak{I},n,n-1,\beta']$ with a scalar matrix $\beta'$. 
By Lemma \ref{lemmasimpleimpliescusp}, $\pi $ is cuspidal. By Lemma \ref{minimalleveltwisting}, $l(\pi)\leqslant l(\chi\pi)$ for every one-dimensional character $\chi: F^{\times}\to\mathbb{C}^{\times}$. 
By Lemma \ref{lemmasimpleimpliescusp}, there exists a simple stratum $[\mathfrak{I},n,n-1,\beta _{1}]$ equivalent to $[\mathfrak{I},n,n-1,\beta]$ such that $\pi \cong \cInd _{J}^{G}\Lambda $ where $J=F[\beta_{1}]^{\times}U_{\mathfrak{I}}^{\lfloor \frac{n+1}{2}\rfloor}$ and $\Lambda\mid _{U_{\mathfrak{I}}^{\lfloor \frac{n}{2}\rfloor+1}}$ contains $\psi _{\beta_{1}}$.
Let $ \rho _{1}$ be an irreducible component of $ \rho \mid _{ U _{ \mathfrak{I}}}$ such that $ \rho _{1} \mid _{ U_{ \mathfrak{I}}^{n}}$ contains $ \psi _{ \beta }\mid_{U_{\mathfrak{I}}^{n}}$. Then 
\begin{equation*} 
\textrm{Hom} _{ U _{ \mathfrak{I}}} \left(  \rho _{1}, \left( \cInd _{J}^{G} \Lambda \right) \mid _{ U _{ \mathfrak{I}}} \right) \neq 0.  
\end{equation*} 
By Mackey formula and Frobenius reciprocity there exists $ g \in J \setminus G / U _{ \mathfrak{I}}$ such that 
\begin{equation} \label{MF}  
\textrm{Hom} _{U _{ \mathfrak{I}}} \left( \rho _{1}, \cInd _{ J^{g} \cap U _{ \mathfrak{I}}}^{ U _{ \mathfrak{I}}} \left( \Lambda ^{g} \mid _{ J^{g} \cap U _{ \mathfrak{I}}} \right) \right) \cong \textrm{Hom } _{J ^{g} \cap U _{ \mathfrak{I}}} \left( \rho _{1} \mid _{ J^{g} \cap U _{ \mathfrak{I}}}, \Lambda ^{g} \mid _{ J^{g} \cap U _{ \mathfrak{I}}} \right) \neq 0.  
\end{equation} 
In particular 
\begin{equation*} 
\textrm{Hom } _{ U _{ \mathfrak{I}}^{n} \cap \left( U _{ \mathfrak{I}} ^{n} \right) ^{g}} \left( \rho_{1} , \Lambda ^{g} \right) \neq 0.  
\end{equation*}     
Denote by $H _{1}$ the subgroup $ U _{ \mathfrak{I}} ^{n}$. There exists $ g_{1} \in U _{ \mathfrak{I}}$ such that 
\begin{equation*} 
\textrm{Hom} _{ H_{1} \cap H_{1}^{g}} \left(\psi _{ \beta }^{g_{1}} \mid _{ H_{1} \cap H_{1}^{g}}, \psi _{ \beta_{1} }^{g} \mid _{ H_{1} \cap H_{1}^{g}} \right) \neq 0. 
\end{equation*} 
Since $g_{1} \in N _{G}(H_{1})= U _{ \mathfrak{I}} \rtimes \langle \Pi _{ \mathfrak{I}} \rangle $ the previous one is equivalent to 
\begin{equation*} 
\textrm{Hom} _{H_{1}^{g_{1}} \cap(H_{1}^{g_{1}})^{g_{1}^{-1}g}} \left( \psi _{ \beta ^{g_{1}}} \mid _{ H _{1} ^{g_{1}} \cap (H_{1} ^{g_{1}})^{ g_{1} ^{-1}g }} , ( \psi _{ \beta_{1} ^{g_{1}}} )^{g_{1}^{-1}g} \mid _{H _{1} ^{g_{1}} \cap (H_{1}^{g_{1}})^{g_{1}^{-1}g}} \right) \neq 0. 
\end{equation*} 
Therefore $g_{1}^{-1} g$ intertwines stratum $[ \mathfrak{I}, n, n-1, \beta_{1} ^{g_{1}}]$. Since $[ \mathfrak{I}, n, n-1, \beta_{1} ]$ is simple and $g_{1} \in U _{ \mathfrak{I}}$, the stratum $[ \mathfrak{I}, n, n-1 , \beta_{1} ^{ g_{1}} ]$ is also simple. Therefore $g \in J U _{ \mathfrak{I}}$. By Lemma \ref{inductioniscuspidaltype}, $\cInd_{J\cap U_{\mathfrak{I}}}^{K}\left(\Lambda\mid_{J\cap U_{\mathfrak{I}}}\right)=\cInd_{ U_{\mathfrak{I}}}^{K}\cInd_{J\cap U_{\mathfrak{I}}}^{U_{\mathfrak{I}}}\left( \Lambda\mid_{J\cap U_{\mathfrak{I}}}\right)$ is irreducible. In particular, $\cInd_{J\cap U_{\mathfrak{I}}}^{U_{\mathfrak{I}}}\left( \Lambda \mid_{J \cap U_{\mathfrak{I}}}\right)$ is also irreducible. By (\ref{MF}), $ \rho _{1}$ is isomorphic to $ \cInd _{ J \cap U _{ \mathfrak{I}}}^{ U_{ \mathfrak{I}}} \left( \Lambda \mid _{ J \cap U _{ \mathfrak{I}}} \right)$. Finally, by Frobenius reciprocity, 
\begin{equation*}\textrm{Hom}_{K}(\rho,\cInd _{J\cap U_{\mathfrak{I}}}^{K}\left( \Lambda\mid_{J\cap U_{\mathfrak{I}}} \right))\cong \textrm{Hom}_{U_{\mathfrak{I}}}(\rho\mid_{U_{\mathfrak{I}}},\cInd _{J\cap U_{\mathfrak{I}}}^{U_{\mathfrak{I}}}\left( \Lambda \mid _{J\cap U_{\mathfrak{I}}}\right)\neq 0. 
\end{equation*}
Therefore $ \rho \cong \cInd _{J \cap U _{ \mathfrak{I}}} ^{ K} \left( \Lambda \mid _{ J \cap U _{ \mathfrak{I}}} \right)$ and $\rho $ is a cuspidal type on $K$ for $\mathfrak{I}( \pi)$.

    \end{proof} 
 \end{subsection}

 \begin{subsection}{Cuspidal types on $\GL _{2}( \mathcal{O} _{F})$}  
 The goal of this subsection is to prove the following theorem:
 \begin{theorem}
\label{TheoremGL2}
A cuspidal type on $K_{2}=\GL_{2}( \mathcal{O}_{F})$ is precisely a one-dimensional twist of one of the following:
\begin{enumerate}
\item a representation inflated from some irreducible cuspidal representation of $\GL_{2}(k_{F})$; 
\item a representation whose orbit contains a matrix whose characteristic polynomial is irreducible mod $\mathfrak{p}_{F}$;
\item a representation whose orbit contains a matrix $\beta _{0} $ whose characteristic polynomial is Eisenstein and which satisfies one of the following: \begin{enumerate}
\item it has conductor at least $4$;
\item it has conductor $r=2$ or $3$ and is isomorphic to $\mathrm{Ind}_{S}^{K_{2}}\theta $ where $S=\bigcup_{a\in\mathcal{O}_{F}^{\times}}
\begin{pmatrix}
a+\mathfrak{p}_{F}& \mathcal{O}_{F}\\
\mathfrak{p}_{F}& a+\mathfrak{p}_{F}
\end{pmatrix}$, $\theta$ is an irreducible character of $S$ such that $\theta \mid _{ U _{ \mathfrak{M}}^{ \lfloor \frac{r+1}{2} \rfloor}}=m \psi _{\varpi_{F}^{-r+1}\beta_{0} }$ for certain $m \in \mathbb{Z}$ and $ \theta $ does not contain the trivial character of the group $\begin{pmatrix}
1&\mathfrak{p}_{F}^{r-2}\\
0&1
\end{pmatrix}$.   
\end{enumerate}  
\end{enumerate}
\end{theorem} 
\begin{remark}[{see subsection \ref{regularity}}]  Any matrix of the form $\Pi _{\mathfrak{I}}B$ for $B \in U_{ \mathfrak{I}}$ has a characteristic polynomial which is Eisenstein. Moreover any matrix whose characteristic polynomial is Eisenstein is $\GL_{p}( \mathcal{O}_{F})$-conjugate to one of the form $\Pi _{ \mathfrak{I}} B$, $B \in U_{\mathfrak{I}}$.  \end{remark}

\begin{lemma}\label{lem-Ciuchacz} Let $\rho $ be a representation with conductor $r\in\{2,3\}$ whose orbit contains a matrix of the form $\Pi_{\mathfrak{I}}B$ for $B\in U_{\mathfrak{I}}$. The representation $\rho $ is a cuspidal type if and only if there exists an irreducible smooth representation $\pi$ of $G$ which contains $\rho $ and whose normalized level is $l(\pi )> r-2$.
\end{lemma} 
\begin{proof}
First we prove that if $\rho$ is a cuspidal type then every irreducible smooth representation $\pi $ of $G$ containing $\rho $ satisfies $l(\pi)>r-2$. For the sake of contradiction, assume that $\rho $ is a cuspidal type and there exists an irreducible smooth representation $\pi _{1}$ of $G$ which contains $\rho $ and has $l(\pi _{1} )\leqslant r-2$. The representation $\rho$ is a cuspidal type for $\mathfrak{I}(\pi _{1})$. 

\textbf{Case }$r=2$, $l(\pi _{1})=0$. By (\cite{BH}, 14.5 Exhaustion theorem), $\pi _{1}\cong\cInd _{K_{2}Z}^{G}\Lambda $ for some $\Lambda $ such that $\Lambda \mid _{K_{2}}$ is inflated from an irreducible cuspidal representation of $\GL_{2}(k_{F})$. By Proposition \ref{BKprop}, $\Lambda\mid_{K_{2}}$ is a cuspidal type for $\mathfrak{I}(\pi _{1})$. By Theorem \ref{Paskunas}, $\Lambda\mid _{K_{2}}\cong \rho$. Since the conductor of $\rho $ is 2 this is impossible. 

\textbf{Case }$r=3$, $l(\pi_1)=0$ is handled as before.

\textbf{Case }$r=3$, $l(\pi_1)=1/2$. By \cite[12.9 Theorem]{BH}, the representation $\pi _{1}$ contains a fundamental stratum $[ \mathfrak{A},1,0,b]$ with $e(\mathfrak{A})=2$. If $\pi _{1}$ contains some stratum, then it contains all of its $G$-conjugates. Therefore $ \pi _{1}$ contains a fundamental stratum $[\mathfrak{I},1,0, a]$ for certain $a$. By \cite[13.1 Proposition 1]{BH}, we have $a\mathfrak{I}=\mathfrak{P}_{ \mathfrak{I}}^{-1}$. In particular, $a \in \mathfrak{K}( \mathfrak{I})$, $\nu _{ \mathfrak{I}}( a)=-1$ so $\varpi_{F}a= \Pi _{ \mathfrak{I}} B$ for some $B \in U_{ \mathfrak{I}}$. By Proposition \ref{Propositionsimplestratum}, the stratum $[ \mathfrak{I}, 1, 0, a]$ is simple. By Lemma \ref{minimalleveltwisting} and Lemma \ref{lemmasimpleimpliescusp}, the representation $\pi _{1}\cong\cInd _{J}^{G}\Lambda$ for certain $J$ containing $U_{\mathfrak{I}}^{1}$ and $\Lambda $ such that $\Lambda\mid_{U_{\mathfrak{I}}^{2}}$ is trivial. By the uniqueness of types $\rho \cong {\rm Ind}_{J\cap K_{2}}^{K_{2}}(\Lambda \mid _{J\cap K_{2}})$ and $\rho \mid _{J\cap K_{2}}$ contains $\Lambda \mid _{J\cap K_{2}}$. Since $U_{\mathfrak{M}}^{2}\subseteq U_{\mathfrak{I}}^{2}$ this means $\rho \mid _{U_{\mathfrak{M}}^{2}}$ is trivial. The conductor of $\rho $ is 3 so this is impossible.

\textbf{Case }$r=3$, $l(\pi_{1})=1$ and for every character $\chi$ of $F^{\times }$ we have $l(\pi _{1})\leqslant l(\chi\pi _{1})$. By Theorem \ref{theoremclassificationcusrep} and \cite[12.9. Theorem]{BH}, $\pi _{1} $ contains a simple stratum of the form $[\mathfrak{A}, e(\mathfrak{A}),e(\mathfrak{A})-1, a]$. We have $U_{\mathfrak{M}}^{2}\subseteq U_{\mathfrak{I}}^{3}$ so arguing like in the previous case we deduce that $\rho \mid _{U_{\mathfrak{M}}^{2}}$ is trivial which again is impossible.

\textbf{Case }$r=3$, $l(\pi_{1})=1$ and there exists a character $\chi$ of $F^{\times }$ such that $l(\chi \pi _{1})=0$ or $l(\chi\pi _{1})=\frac{1}{2}$. Applying the arguments from before to $\chi\rho$ and $\chi\pi _{1}$ we deduce that $\chi\rho \mid _{U_{\mathfrak{I}}^{2}}$ is trivial. 
In particular, both $\chi\rho$ and $\rho$ are trivial on $U^3_\mathfrak{M}$. Therefore $\chi\circ{\rm det}\mid _{U_{\mathfrak{M}}^{3}}$ is trivial. By Lemma \ref{lem-Gutoslaw} this is impossible.
\begin{lemma}\label{lem-Gutoslaw}
Let $\rho $ be an irreducible smooth representation of $K$ with conductor $r>1$ and an orbit containing a matrix whose characteristic polynomial is Eisenstein. Let $\chi $ be a character of $F^{\times }$ such that $\chi\circ {\rm det}\mid _{U_{\mathfrak{M}}^{r}}$ is trivial. Then the conductor of $\chi\rho $ is greater than or equal to $r$. 
\end{lemma}
\begin{proof}
Let $\lbrace \beta _{i} \rbrace _{i\in I}$ be the orbit of $\rho $. 
By Lemma \ref{determinantscalar}, $\chi\circ {\rm det}\mid _{U_{\mathfrak{M}}^{r-1}}$ is of the form $\psi _{\alpha }$ for some scalar matrix $\alpha \in \mathfrak{P}_{\mathfrak{M}}^{-r+1}$. 
By the definitions of $\psi _{\beta _{i} }$ and $\psi_{\alpha }$, 
\begin{center}
$\chi\rho\mid _{U_{\mathfrak{M}}^{r-1}}= m\bigoplus  _{i\in I}\psi _{\varpi_{F}^{-r+1}\beta _{i}+\alpha }\mid _{U_{\mathfrak{M}}^{r-1}}$ for some $m\in\mathbb{N}$.
\end{center}The restriction $\chi\rho\mid _{U_{\mathfrak{M}}^{r-1}}$ is trivial if and only if $\varpi_{F}^{-r+1}\beta _{i}+\alpha\in \mathfrak{P}_{\mathfrak{M}}^{-r+2}$ for all $i\in I$. If $\chi\rho|_{U_{\mathfrak{M}}^{r-1}}$ is trivial, then $\beta _{1}+ \varpi _{F}^{r-1}\alpha \in \mathfrak{P}_{\mathfrak{M}}$. This is impossible because the matrices in $\M_2(\mathcal O_F)$ whose characteristic polynomial is Eisenstein are not congruent to a scalar matrix modulo $\mathfrak P_\mathfrak M$.

    
\end{proof}

Now we prove the converse: if there exists an irreducible smooth representation of $G$ which contains $\rho $ and whose normalized level is strictly greater than $r-2$, then $\rho $ is a cuspidal type. Assume that there exists an irreducible representation $\pi $ of $G$ which contains $\rho $ and such that $l(\pi )>r-2$. The representation $\rho $ has an orbit containing a matrix of the form $\alpha _{0}=\Pi _{ \mathfrak{I}} B$ with $B \in U_{ \mathfrak{I}}$. Let $\alpha :=\varpi_{F} ^{-r+1}\alpha _{0}$. In particular, $\rho $ contains the character $\psi _{\alpha }$ of $U_{\mathfrak{M}}^{\lfloor\frac{r+1}{2}\rfloor }=U_{\mathfrak{M}}^{r-1}$. We want to show that $\psi _{\alpha }$ has an extension to $U_{\mathfrak{I}}^{2r-3}$ which is trivial on $U_{\mathfrak{I}}^{2r-2}$ and which is contained in $\rho $. Since $\alpha \mathfrak{P}_{\mathfrak{I}}^{2r-2} \subseteq \mathfrak{P}_{ \mathfrak{I}}$, $\psi _{\alpha }$ is trivial not only on $U_{\mathfrak{M}}^{r}$ but also on $U_{\mathfrak{I}}^{2r-2}$. Therefore $\psi _{\alpha }$ can be seen as a character of $U_{\mathfrak{M}}^{r-1}/U_{\mathfrak{I}}^{2r-2}$. We have $U_{\mathfrak{M}}^{r-1}/U_{\mathfrak{I}}^{2r-2}\subseteq U_{\mathfrak{I}}^{2r-3}/U_{\mathfrak{I}}^{2r-2}$ and the last group is abelian, so $\psi _{\alpha }$ has a one-dimensional extension to $U_{\mathfrak{I}}^{2r-3}/U_{\mathfrak{I}}^{2r-2}$ which is contained in $\rho$. This extension is of the form $\psi_{\kappa }$ for $\kappa\in \mathfrak{P}_{\mathfrak{I}}^{-2r+3}$. This means that the stratum $[\mathfrak{I},2r-3, 2r-4, \kappa ]$ is contained in $\pi $ and $\psi _{\kappa }$ is contained in $\rho $. Since the normalized level of $\pi$ is strictly greater than $r-2$, it has to be equal to $\frac{2r-3}{2}$. By (\cite{BH}, 12.9 Theorem) the stratum $[\mathfrak{I},2r-3, 2r-4, \kappa ]$ is fundamental. By \cite[13.1. Proposition 1.]{BH}, it is also simple. Similarly as in the Case 4.2 of the proof of Theorem \ref{TheoremGLp}, this means that $\rho $ is a cuspidal type. This concludes the proof of Lemma \ref{lem-Ciuchacz}.
\end{proof}  
 \begin{proof}[Proof of Theorem \ref{TheoremGL2}]
By Theorem \ref{TheoremGLp}, to prove Theorem \ref{TheoremGL2} it is enough to prove that a representation of $K_{2}$ whose conductor is $2$ or $3$ and whose orbit is equivalent to an orbit which contains a matrix of the form $\Pi _{\mathfrak{I}}B$ for $B \in U_{\mathfrak{I}}$ is a cuspidal type if and only if 3(b) from Theorem \ref{TheoremGL2} is satisfied. 

Let $\rho $ be a representation of conductor $r=2$ or $r=3$ whose orbit contains a matrix $\beta_{0}=\Pi _{ \mathfrak{I}}B$ for $B \in U_{ \mathfrak{I}}$. Denote $\beta:=\varpi_{F}^{-r+1}\beta_{0}$. 
By Remark \ref{rho}, the representation $\rho $ is isomorphic to $\mathrm{Ind } _{\mathrm{Stab}_{K_{2}}(\psi _{\beta} )}^{K_{2}}\theta $ for some irreducible representation $\theta$ of $\mathrm{Stab}_{K_{2}}(\psi _{\beta})$ such that $\theta \mid_{U_{\mathfrak{M}}^{\lfloor\frac{r+1}{2}\rfloor}}$ is a multpile of $\psi_{\beta}$. Note that $\lfloor\frac{r+1}{2}\rfloor=r-1$ for $r=2$ or $r=3$. Denote $S:=\textrm{Stab}_{K_{2}}(\psi_{\beta})$. By \cite[section 2]{S}, 
\begin{center}
$S=\bigcup _{a\in \mathcal{O}_{F}^{\times }}\begin{pmatrix}
a+\mathfrak{p}_{F}&\mathcal{O}_{F}\\
\mathfrak{p}_{F}&a+\mathfrak{p}_{F}
\end{pmatrix}$.
\end{center}  
By Lemma \ref{lem-Ciuchacz}, to prove Theorem \ref{TheoremGL2} it is enough to show the following claim. 

\textbf{Claim.} \textit{The representation $\theta $ contains the trivial character of the subgroup 
$$H_0:=\begin{pmatrix}
1&\mathfrak{p}_{F}^{r-2}\\
0&1
\end{pmatrix}$$ if and only if every irreducible smooth representation of $G$ containing $\rho $ has normalized level less than or equal to $r-2$. }
\\

First we show that if $\theta$ contains the trivial character of $H_0$, then the normalized level of any irreducible smooth representation $\pi$ of $G$ containing $\rho$ is at most $r-2$. Up to $\GL_{2}(\mathcal{O}_{F})$-conjugation 
$$\beta _{0}\in \begin{pmatrix}
\mathfrak{p}_{F}&1+\mathfrak{p}_{F}\\
\mathfrak{p}_{F}&\mathfrak{p}_{F}
\end{pmatrix}.$$ Therefore we can assume that the character $\psi _{\beta}$, and consequently $\theta$, are trivial when restricted to $$H_1:=\begin{pmatrix}
1+\mathfrak{p}_{F}^{r-1}&\mathfrak{p}_{F}^{r-1}\\
0&1+\mathfrak{p}_{F}^{r-1}
\end{pmatrix}.$$ 
The groups  $H_0$, 
  $H_1$ and $U_\mathfrak{M}^r$ generate  
the group $$\begin{pmatrix}
1+\mathfrak{p}_{F}^{r-1}&\mathfrak{p}_{F} ^{r-2}\\
\mathfrak{p}_{F}^{r}&1+\mathfrak{p}_{F}^{ r-1}
\end{pmatrix}=\begin{pmatrix}
0&1\\
\varpi_{F} &0
\end{pmatrix}^{-1}U_\mathfrak{M}^{r-1}\begin{pmatrix}
0&1\\
\varpi_{F} &0
\end{pmatrix},$$ 
so $\theta $ contains the trivial character of a $\GL_{2}(\mathcal{O}_{F})$-conjugate of $U_\mathfrak{M}^{r-1}.$ Every irreducible smooth representation $\pi$ of $G$ containing $\rho$ also contains $\theta$ and consequently the trivial character of $U_\mathfrak{M}^{r-1}$. This means that the normalized level of $\pi$ is at most $r-2$. 

Now assume that every irreducible smooth representation of $G$ which contains $\rho $ has normalized level less than or equal to $r-2$. By (\cite[11.1 Proposition 1]{BH}), there exists $g \in G$ such that \begin{center}
$\mathrm{Hom}_{(U_{\mathfrak{M}}^{r-1})^{g}\cap S}(1_{(U_{\mathfrak{M}}^{r-1})^{g}\cap S}, \theta)\neq 0$.
\end{center}  
Therefore $\theta $ contains the trivial character of $(U_{\mathfrak{M}}^{r-1})^{g}\cap S$. This property depends only on the double coset $N_{G}(U_{\mathfrak{M}}^{r-1})gS$, where $N_{G}(U_\mathfrak{M}^{r-1})$ denotes the normalizer of $U_\mathfrak{M}^{r-1}$ in $G$. 
By \cite[12.3]{BH}, $N_{G}(U_{\mathfrak{M}}^{r-1})=K_{2}Z$.  By \cite[17.1 Proposition]{BH},
\begin{equation*}
G=\bigcup_{m\in\mathbb{Z}}K_{2}Z
\begin{pmatrix}
\varpi_{F}^{m}&0\\
0&1
\end{pmatrix}
\begin{pmatrix}
\mathcal{O}_{F}^{\times}&\mathcal{O}_{F}\\
\mathfrak{p}_{F}&\mathcal{O}_{F}^{\times}\end{pmatrix}.
\end{equation*}
Since $\begin{pmatrix} \mathcal{O}_{F}^{ \times} & \mathcal{O}_{F} \\ 
\mathfrak{p}_{F} & \mathcal{O}_{F}^{ \times } \end{pmatrix}= \bigcup _{c \in \mathcal{O} _{F} ^{ \times }} \begin{pmatrix} c&0\\ 
0&1 
\end{pmatrix} S$, we can assume that $g$ is of the form either $g_{1,n,c}=\begin{pmatrix}
1&0\\
0&c\varpi_{F}  ^{n}
\end{pmatrix}$ or $g_{2,n,c}=\begin{pmatrix}
c\varpi_{F} ^{n}&0\\
0&1
\end{pmatrix}$ for some $c\in \mathcal{O}_{F}^{\times }$ and $n\in \mathbb{N}$. 
\\
Compute \begin{center}
$(U_{\mathfrak{M}}^{r-1})^{g_{1,n,c}}\cap S= \begin{pmatrix}
1+\mathfrak{p}_{F}^{r-1}&\mathfrak{p}_{F}^{r-1+n}
\\
\mathfrak{p}_{F}^{r-1-n}\cap \mathfrak{p}_{F}&1+\mathfrak{p}_{F}^{r-1}
\end{pmatrix}$
\end{center}   
\begin{center}
$(U_{\mathfrak{M}}^{r-1})^{g_{2,n,c}}\cap S= \begin{pmatrix}
1+\mathfrak{p}_{F}^{r-1}&\mathfrak{p}_{F}^{r-1-n}\cap \mathcal{O}_{F}\\
\mathfrak{p}_{F}^{r-1+n}&1+\mathfrak{p}_{F}^{r-1}
\end{pmatrix}$.
\end{center} In particular, $\theta $ contains the trivial character of at least one of the following:
\begin{enumerate}
\item $\bigcap _{n\geqslant 0, c\in \mathcal{O}_{F}^{\times }}(U_{\mathfrak{M}}^{r-1})^{g_{1,n,c}}\cap \bigcap _{c\in \mathcal{O}_{F}^{\times }}(U_{\mathfrak{M}}^{r-1})^{g_{2,0,c}}\cap  S=\begin{pmatrix}1+\mathfrak{p}_{F}^{ r-1}&0\\
\mathfrak{p}_{F}^{ r-1}&1+\mathfrak{p}_{F}^{ r-1}
\end{pmatrix}
$
\item $\bigcap _{n>0,c\in \mathcal{O}_{F}^{\times }}(U_{\mathfrak{M}}^{ r-1})^{g_{2,n,c}}\cap S=\begin{pmatrix} 1+\mathfrak{p}_{F}^{ r-1} & \mathfrak{p}_{F} ^{ r-2}\\
0&1+\mathfrak{p}_{F}^{ r-1}
\end{pmatrix}.$ 
\end{enumerate}
The restriction $\theta\mid _{U_{\mathfrak{M}}^{ r-1}}$ is a multiple of $\psi _{\beta }$. Since  
\begin{equation*}\psi _{\beta }\left( \begin{pmatrix}
1+\mathfrak{p}_{F}^{ r-1}&0\\
\mathfrak{p}_{F}^{ r-1}&1+\mathfrak{p}_{F}^{ r-1}
\end{pmatrix}\right) =\psi (\mathcal{O}_{F})\neq \{1\},
\end{equation*} $\theta $ does not contain the trivial character of 
$\begin{pmatrix} 
1+ \mathfrak{p}_{F}^{ r-1}&0\\ 
\mathfrak{p}_{F}^{ r-1} & 1+ \mathfrak{p}_{F}^{r-1} 
\end{pmatrix}$. Therefore $\theta $ contains the trivial character of $\begin{pmatrix}
1+\mathfrak{p}_{F}^{r-1}&\mathfrak{p}_{F} ^{r-2}\\
0&1+\mathfrak{p}_{F}^{r-1}
\end{pmatrix}$.  
 \end{proof}
\begin{corollary}[see subsection \ref{regularity} and section \ref{sectionexample}]
Every cuspidal type on $K_{2}$ is a regular representation. However not every regular representation of $K_{2}$ is a cuspidal type on $K_{2}$.
\end{corollary} 
\end{subsection} 

\begin{subsection}{Regularity of cuspidal types} \label{regularity} In this subsection we determine which cuspidal types are regular. Matrices whose characteristic polynomial is irreducible modulo $\mathfrak{p}_{F}$ are regular so we focus on matrices of the form $\Pi_{\mathfrak{I}}^{j}B$ where $j\in\mathbb{N}$, $0<j<p$ and $B\in U_{\mathfrak{I}}$. In this subsection we work with matrices of arbitrary dimension $n\geqslant 2$, not necessarily prime. 

Fix a natural number $n\geqslant 2$. Define $ \mathfrak{I}$, $\Pi _{ \mathfrak{I}}$ analogously as in subsection \ref{chain1} but for the dimension $n$.  

\begin{lemma}\label{conjugation}
Let $M\in \M_{n}(\mathcal{O} _{F})$ be a matrix whose characteristic polynomial is Eisenstein. Then any $g\in \GL_{n}(F)$ such that $g^{-1}Mg\in \M _{n}(\mathcal{O}_{F})$ is of the form $g\in Z_{\GL_{n}(F)}(M)\GL_{n}(\mathcal{O} _{F})$ where $Z_{\GL_{n}(F)}(M)$ denotes the centralizer of $M$ in $\GL_{n}(F)$. In particular, there exists an element $h\in \GL_{n}(\mathcal{O} _{F})$ such that $h^{-1}Mh=g^{-1}Mg$. \end{lemma}
\begin{proof}Take the lattice $\Lambda =\mathcal{O}_{F} ^{n}$. For any matrix $Q\in \GL_{n}(F)$ the following holds \begin{equation}\label{warunek} Q\Lambda =\Lambda\quad  \text{if and only if}\quad Q\in \GL_{n}(\mathcal{O} _{F}).\end{equation} Therefore to prove Lemma \ref{conjugation} it is enough to show that for any $g\in\GL_{n}(F)$ such that $g^{-1}Mg\in\textrm{M}_{n}(\mathcal{O}_{F})$ we have $g\Lambda=z \Lambda $ for some $z\in Z_{\GL_{n}(F)}(M)$. 

  Let $E:=F[M]$. Since the characteristic polynomial of $M$ is irreducible, $E$ is a field. The action of $M$ on $F^{n}$ turns it into a one-dimensional $E$-vector space. Fix a non-zero element $v\in F^{n}$. The following map is an isomorphism of $E$-vector spaces $i:E\ni x\mapsto xv\in F^{n}$. 
 
We want to show that $\Lambda$ and $g\Lambda $ are $\mathcal{O}_{E}$-modules.
 Both $\Lambda$ and $g\Lambda$ are $\mathcal{O}_{F}[M]$-modules so it is enough to show that $\mathcal{O}_{E}=\mathcal{O} _{F}+\mathcal{O} _{F}M+\ldots +\mathcal{O} _{F}M^{n-1}$. Since the residue field of $E$ and of $F$ are the same, we have $\mathcal{O}_{E}=\mathcal{O} _{F}+\mathcal{O}_{E}M$. Therefore $\mathcal{O}_{E}=\mathcal{O}_{F}+\mathcal{O}_{F}M+\ldots+\mathcal{O}_{F}M^{n-1}+\mathcal{O}_{E}M^{n}$. By Nakayama's lemma 
 $\mathcal{O}_{E}=\mathcal{O} _{F}+\mathcal{O} _{F}M+\ldots +\mathcal{O} _{F}M^{n-1}$.   

Recall that fractional ideals of $E$ are finitely generated $O_{E}$-submodules of $E$. Since $i$ is $E$-linear there exist fractional ideals $I_{1}$ and $I_{2}$ of $E$ such that $i(I_{1})=\Lambda$ and $i(I_{2})=g\Lambda$. The fractional ideals of $E$ are generated by powers of $M$, so there exists an integer number $j$ such that $g\Lambda =M^{j}\Lambda$. 
 \end{proof}   
 
\begin{proposition} \label{PropositionEisensteinregular}Let $M\in \M _{n}(\mathcal{O} _{F})$ be a matrix whose characteristic polynomial $f$ is Eisenstein. Then $M$ is $\GL_{n}(\mathcal{O} _{F})$-conjugate to the companion matrix of $f$, which is regular. In particular, $M$ is regular.\end{proposition}
\begin{proof}Denote by $C$ the companion matrix of $f$. 
By the Skolem--Noether theorem there exists $g\in \GL_{n}(F)$ such that $M=gCg^{-1}$. By Lemma \ref{conjugation}, matrices $M$ and $C$ are $\GL_{n}(\mathcal{O} _{F})$-conjugate. 
  

\end{proof}Observe that the characteristic polynomial $f(x)$ of $\Pi _{ \mathfrak{I}} B$, with $B \in U _{ \mathfrak{I}}$, is equal to $x^{p}$ modulo $\mathfrak{p}_{F}$ and $f(0)= {\rm det} ( \Pi _{ \mathfrak{I}} B)= u \varpi _{F}$ for some $u \in \mathcal{O}_{F}^{ \times }$. Therefore $f(x)$ is Eisenstein and Proposition \ref{PropositionEisensteinregular} proves that the matrices of the form $ \Pi _{ \mathfrak{I}} B$ with $B\in U_{\mathfrak{I}}$ are regular. 
\begin{proposition}\label{Propositionnotregular}Matrices of the form $\Pi _{ \mathfrak{I}}^{j} B$ for $j \in \mathbb{N}$, $1<j<p$ and $B \in U_{ \mathfrak{I}}$ are not regular.  
\end{proposition} 
\begin{proof} Denote by $D$ the reduction of $\Pi _{ \mathfrak{I}}^{j}B$ modulo $\mathfrak{p}_{F}\M_{n}(\mathcal{O}_{F})$. For the sake of contradiction assume that $D$ is regular. Denote by $\bar{ \Pi}_{ \mathfrak{I}}$ the reduction of $\Pi _{ \mathfrak{I}}$ modulo $ \mathfrak{p}_{F}\M_{n}(\mathcal{O}_{F})$. The matrix $D$ is nilpotent. By \cite[14.16 Proposition]{MD}, a regular nilpotent matrix in $\M_{n}(k_{F})$ is $\GL _{n}( k_{F})$-conjugate to $\bar{ \Pi }_{ \mathfrak{I}}$. This is impossible because $D^{n-1} = 0$ and $\left( \bar {\Pi }_{\mathfrak{I}}\right)^{n-1}\neq 0$.
\end{proof} 
 \end{subsection}
\end{section}

\begin{section}{Example} \label{sectionexample} 

In this section we give an example of two representations of $K_{2}=\GL_{2}( \mathcal{O}_{F})$ with the same orbits and the same conductor but one of them will be a cuspidal type and the second will not. This illustrates the fact that it is not always enough to know the orbits to determine if a representation is a cuspidal type. Let $$S=\bigcup _{ a \in \mathcal{O}_{F}^{ \times}} \begin{pmatrix}a+ \mathfrak{p} _{F}& \mathcal{O}_{F}\\ \mathfrak{p}_{F} & a+ \mathfrak{p}_{F} \end{pmatrix},\beta _{1} =\varpi _{F} ^{-1} \begin{pmatrix} 0&1\\ \varpi _{F}& 0 \end{pmatrix}\text{ and }\beta _{2} = \varpi _{F}^{-1} \begin{pmatrix} 0&1\\0&0 \end{pmatrix}.$$ Define $\theta _{1}, \theta _{2}: S \rightarrow \mathbb{C}^\times$ as follows: 
\begin{equation*} \theta _{1}(a {\rm Id} _{2 \times 2} +x)= \psi(\tr (a^{-1}\beta _{1} x))\quad \text{and}\quad \theta _{2}(a {\rm Id} _{2 \times 2}+x)= \psi ( \tr (a^{-1} \beta _{2}x))\end{equation*} where $a \in \mathcal{O}_{F}^{\times}$ and $x \in \begin{pmatrix} \mathfrak{p}_{F} & \mathcal{O}_{F}\\ \mathfrak{p}_{F}& \mathfrak{p}_{F} \end{pmatrix}$.
Simple computation shows that $\theta_1$ and $\theta_2$ are well defined characters of $S$. 
Let \begin{equation*} \rho _{1}:= {\rm Ind} _{S}^{K_{2}} \theta _{1} \quad \text{and} \quad \rho _{2}:= {\rm Ind} _{S}^{K_{2}} \theta _{2}. \end{equation*} \

\begin{proposition}Both representations $\rho_{1}$ and $\rho _{2}$ have conductor $2$ and contain the matrix $\beta _{0}:= \varpi _{F} \beta _{1} $ in their orbits. The representation $\rho _{1}$ is a cuspidal type but $ \rho _{2}$ is not. 
\end{proposition}   
\begin{proof} 
Both $\rho _{1}|_{U_{ \mathfrak{M}}^{1}}$ and $\rho _{2}| _{U_{ \mathfrak{M}}^{1}}$ contain $\psi _{ \beta _{1}}=\theta_1|_{U_{\mathfrak{M}}^1}=\theta_2|_{U_{\mathfrak{M}}^1}$, so $\rho _{1}$ and $\rho _{2}$ have conductor $2$ and contain $\beta_0=\varpi _{F} \beta _{1} $ in their orbits.

We prove that $\rho _{2}$ is not a cuspidal type. For the sake of contradiction assume that $\rho _{2}$ is a cuspidal type for an irreducible cuspidal representation $\pi $ of $\GL_{2}(F)$. Assume $l(\pi)=0$. By Theorem \ref{theoremclassificationcusrep},  there exists a representation $\Lambda$ of $ZK$ such that $\Lambda\mid_{K}$ is trivial on $U_{\mathfrak{M}}^{1}$ and $\pi \cong \cInd_{ZK}^{G}\Lambda$. As before $\rho_{2}\cong \Lambda\mid_{K}$. This is impossible becasue $\rho_{2}$ is of conductor $2$. Therefore $l(\pi)>0$. We have 
\begin{equation*}\theta _{2} | _{\begin{pmatrix} 1+ \mathfrak{p}_{F} & \mathcal{O} _{F} \\0& 1+ \mathfrak{p}_{F} \end{pmatrix}}=1\ \ \ \textrm{and} \ \ \ \theta _{2} |_{ U_{ \mathfrak{M}}^{2}}=1.
\end{equation*}
The groups $\begin{pmatrix} 1+ \mathfrak{p}_{F} &\mathcal{O}_{F} \\ 0& 1+ \mathfrak{p}_{F} \end{pmatrix}$and $U_{ \mathfrak{M}}^{2}$ generate $\begin{pmatrix} 1+ \mathfrak{p}_{F} & \mathcal{O}_{F} \\ \mathfrak{p}_{F}^{2} &1+ \mathfrak{p}_{F} \end{pmatrix} $. The last group is $\GL_{2}(F)$-conjugate to $U_{ \mathfrak{M}}^{1}$. Therefore $l( \pi )=0$. This leads to a contradiction.

We have $\textrm{Stab}_{K_{2}}(\bar{\varphi}_{\bar{\beta_{0}}})=S$ and $\theta_{1}$ is not trivial on 
$\begin{pmatrix} 
1&\mathfrak{p}_{F}\\
0&1
\end{pmatrix}$. By Theorem \ref{TheoremGL2}, $\rho _{1}$ is a cuspidal type.        

\end{proof}  
    
\end{section} 
\bibliography{ref} 
\bibliographystyle{plain}   
\end{document}